%% file: paper1_revision.tex
\documentclass[12pt,a4paper]{article}
\usepackage{graphicx}
\usepackage{amsmath,amssymb,amsthm}
\usepackage{paralist}
\usepackage{float}
\usepackage{overpic}
\usepackage{pdflscape}
\usepackage{tikz}

\usepackage[
    backend=biber,
    style=numeric,
    sorting=none,
    url=false,
    clearlang=true,
    isbn = false,
]{biblatex}
\AtEveryBibitem{\clearfield{pagetotal}}					% removes pagetotal from bibliography
\AtEveryBibitem{\clearfield{note}}						% removes notes from bibliography
\DeclareRedundantLanguages{English,english,german,french, eng}{english,german,ngerman,french}
\renewbibmacro{in:}{%
	\ifentrytype{article}{}{\printtext{\bibstring{in}\intitlepunct}}}		% remove In. in front of Journal bei @article 
	
\usepackage{standalone}
\usepackage{tikz}
\usetikzlibrary{decorations.pathmorphing}
\usetikzlibrary{decorations.markings}
\usetikzlibrary{arrows}
\usetikzlibrary{fit}

\usepackage[hidelinks]{hyperref}
\usepackage[nameinlink]{cleveref}

\newtheorem{theorem}{Theorem}[section]
\theoremstyle{definition}
\newtheorem{definition}[theorem]{Definition}
\theoremstyle{remark}
\newtheorem{remark}[theorem]{Remark}
\newtheorem{example}[theorem]{Example}

\newcommand{\R}{\mathbb{R}}

\newcommand{\Z}{\mathbb{Z}}

\newcommand{\G}{\mathbf{G}}
\newcommand{\Net}{\mathbf{N}}
\newcommand{\GG}{\mathbf{\Gamma}}

\newcommand{\X}{\mathbb{X}}
\newcommand{\XX}{\mathbf{X}}

\begin{document}

%%%
\begin{center}
{\large{\bf 
Design of Hierarchical Excitable Networks}}\\
\mbox{} \\
\begin{tabular}{cc}
{\bf Sören von der Gracht$^{\dagger}$} & {\bf Alexander Lohse$^{\ddagger}$} \\
{\small soeren.von.der.gracht@uni-paderborn.de} & {\small alexander.lohse@uni-hamburg.de}
\end{tabular}

\end{center}

\noindent $^{\dagger}$ Institute of Mathematics, Paderborn University, Warburger Str.~100, 33098 Paderborn, Germany. \\ ORCiD: 0000-0002-8054-2058

\noindent $^{\ddagger}$ Department of Mathematics, University of Hamburg, Bundesstra{\ss}e~55, 20146 Hamburg, Germany \\ ORCiD: 0000-0002-9834-181X

\begin{abstract}
We provide a method to systematically construct vector fields for which the dynamics display transitions corresponding to a desired hierarchical connection structure. This structure is given as a finite set of directed graphs $\G_1,\dotsc,\G_N$ (the lower level), together with another digraph $\GG$ on $N$ vertices (the top level). The dynamic realizations of $\G_1,\dotsc,\G_N$ are heteroclinic networks and they can be thought of as individual connection patterns on a given set of states. Edges in $\GG$ correspond to transitions between these different patterns. In our construction, the connections given through $\GG$ are not heteroclinic, but excitable with zero threshold. This describes a dynamical transition between two invariant sets where every $\delta$-neighborhood of the first set contains an initial condition with $\omega$-limit in the second set. Thus, we prove a theorem that allows the systematic creation of hierarchical networks that are excitable on the top level, and heteroclinic on the lower level. Our results modify and extend the simplex realization method by Ashwin \& Postlethwaite.
\medskip
\end{abstract}

\noindent {\em Keywords:} heteroclinic network, heteroclinic cycle, excitable network, directed graph, hierarchical network

\vspace{.3cm}

\noindent {\em AMS classification:} 34C37, 37C80, 37C75
\vspace{3cm}

%%%%
\section{Introduction}
Heteroclinic cycles and networks are invariant sets that are associated with intermittent behavior in a dynamical system, where typical trajectories recurrently spend long periods of time near different saddle states~\cite{CastroLohse.2026b}. These structures can display intricate stability configurations and subtly influence the overall dynamics of the system. As such they have been studied for several decades, see~\cite{Krupa1997, Weinberger.2018} for two review papers. An excitable connection is another type of dynamical transition that produces similar behavior, and differs from a heteroclinic connection only in backwards time.

From a modeling perspective heteroclinic/excitable structures are adequate mathematical objects for producing intermittent system behavior corresponding to a fixed connection pattern. Examples include neuroscience~\cite{Ortmanns2023}, game theory~\cite{Castroetal.2022} and the study of coupled oscillators~\cite{BickLohse.2019}. In some cases a more fine-grained understanding of the transitions between different states is needed to fully grasp their mechanisms and function---e.g., due to the presence of an intrinsic hierarchy in the sequential processes, where modulations typically occur on a ``lower level'' until the behavior changes substantially, caused by another process on a ``higher level''. Such hierarchical transitions have been reported to play a crucial role in numerous applications (with no claim to completeness of the following list): 
In neuroscience, hierarchy has been identified as a key feature for a multitude of cognitive processes, such as creativity, memory formation, and learning, see~\cite{Ortmanns2023,Rabinovich.2020,KoksalErsoz.2025} and references therein.
In mathematical biology, it can cause polyrythms in central pattern generators believed to steer animal movement \cite{Wojcik.2014}.
Temporal modulation of the possible transitions between different states  according to a higher-level process is also relevant for social and biological networks with competition dynamics~\cite{Holme.2015}. It may indicate the temporal modification of the network itself, for example due to physical proximity of mobile individuals or due to interaction restrictions via temporally evolving natural environments.

In recent years, a complementary question to observing heteroclinic or excitable structures in given systems has drawn increased attention: How to construct a dynamical system that possesses a desired connection structure that is prescribed by a given directed graph (digraph)? This is commonly referred to as \emph{realizing the digraph as a heteroclinic/excitable cycle or network}. Different methods have been derived to achieve this under varying assumptions on the initial digraph~\cite{AshPos2013,Ashwin.2016,Field_2015,Field_2017}, and refined with a focus on properties of the resulting networks~\cite{AshCasLoh2020, CastroLohse.2026a} as well as efficiency in terms of the required state space dimension~\cite{CastroLohse.2023, GarridodaSilva.2026}. Systematic obstructions to the emergence of desired heteroclinic connections have been investigated in \cite{Bick.2024b}.

In this paper, we specifically address the hierarchical aspect in the context of this question by introducing a construction method that realizes given connection structures with an inherent hierarchy. We start with a finite set of digraphs $\G_1,\dotsc, \G_N$, the lower level of the hierarchy, and another digraph $\GG$ on $N$ vertices which corresponds to the top level. Our construction produces heteroclinic networks $\Net_1,\ldots, \Net_N$ realizing $\G_1,\dotsc, \G_N$ in bounded regions of suitable subspaces of a high-dimensional phase space. This is done by adapting one of the established techniques, the simplex realization method from~\cite{AshPos2013}. In our system, the top-level transitions are created not as heteroclinic, but excitable connections~\cite{Ashwin.2016} between the $\Net_j$: Unlike a heteroclinic connection, an excitable connection from $\Net_i$ to $\Net_k$ passes through a small neighborhood of, but does not necessarily limit to $\Net_i$ in backwards time---in fact, the $\alpha$-limit set may be empty. In forward time, such a connection converges to $\Net_k$ and causes nearby dynamics which are phenomenologically indistinguishable from their heteroclinic counterpart \Cref{fig:het-ex}. In our system, excitable connections from $\Net_i$ to $\Net_k$ exist in any $\delta$-neighborhood of $\Net_i$, this is called a connection with threshold zero.
Examples for similar hierarchical dynamics have been investigated in the literature, see e.g.~\cite{VoitMeyerOrt2018} and \cite{Bick.2026} for a network of coupled phase oscillators, but to our knowledge, this is the first time a systematic construction process is proposed.

Hierarchical transition structures in phase space are related to the concept of depth for a heteroclinic connection that was introduced in~\cite{AshFie1999}: Roughly speaking, a heteroclinic connection is typically of depth one if its $\alpha$- and $\omega$-limit sets both consist of a single equilibrium. A connection of higher depth limits, for example, to a heteroclinic cycle (in forward and/or backward time), see~\Cref{fig:depth2}. 

The rest of this paper is structured as follows: In~\Cref{sec:prelim}, we recall the precise definitions of heteroclinic/excitable connections and their relation. We also briefly review the simplex method for realizing digraphs as heteroclinic networks by Ashwin and Postlethwaite~\cite{AshPos2013}. \Cref{sec:main} contains our main result: We define the term \emph{excitable hierarchical network} and prove a realization result for our construction method. In~\Cref{sec:numerics}, we discuss two examples in detail and show numerical simulations that illustrate our results and lead to additional conjectures. We conclude with a discussion in~\Cref{sec:discussion}.

%%%%
\section{Preliminaries}
\label{sec:prelim}
Throughout this paper we are concerned with dynamical systems
\begin{align}
    \label{sys1}
\dot x = f(x)
\end{align}
where $x \in \R^d$ and $f: \R^d \to \R^d$ is a vector field that generates a smooth flow.
We first recall some necessary background on heteroclinic and excitable connections between equilibria of system~\eqref{sys1} and extend it to transitions between invariant sets in phase space. Then we review the simplex method from \cite{AshPos2013}, which provides a way of constructing dynamical systems with heteroclinic connections corresponding to a prescribed directed graph. As we generalize this method in the remainder of this paper, this section also serves to fix important notation.

\subsection{Excitable and heteroclinic connections}
\label{subsec:ex-het}
Suppose that system~\eqref{sys1} possesses two hyperbolic equilibria $\xi_1, \xi_2 \in \R^d$.
If the intersection between their respective unstable and stable manifolds is non-empty, i.e.,
\[W^u(\xi_1) \cap W^s(\xi_2) \neq \emptyset,\]
then any trajectory in this intersection converges to $\xi_2$ in forward time, and to $\xi_1$ in backwards time. Such a solution is called a \emph{heteroclinic connection from $\xi_1$ to $\xi_2$}, see the left panel in \Cref{fig:het-ex}, and we often denote it by $[\xi_1 \to \xi_2]$. It is well-known that heteroclinic connections can form cycles and networks (connected unions of several cycles) that are robust to perturbations within a certain class of dynamical systems---two of the most prominent cases are equivariant systems and Lotka-Volterra systems, see e.g.~\cite{Weinberger.2018} for an overview.

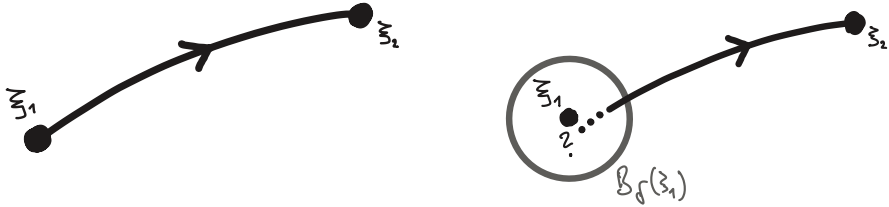
\begin{figure}
    \centering
	\begin{tikzpicture}[line cap=round, line join=round, >=stealth]
		%Heteroclinic [\xi_1 \to \xi_2]
		\filldraw[black] (-4.7,0) circle (0.1) node[font=\Large, left=3pt] {$\xi_1$};
		\filldraw[black] (-0.55,1.75) circle (0.1) node[font=\Large, right=3pt] {$\xi_2$};
		\draw[black, line width=2.8pt, postaction={decorate}, decoration={markings, mark=at position 0.5 with {\arrow[black,line width=2pt, scale=1.8]{>}}}]
		(-4.7,0)  .. controls (-3.2,1.15) and (-1.7,1.55) .. (-0.55,1.75);
		
		%Excitable [\xi_1 \to \xi_2]
		\draw[gray!90, line width=2.2pt] (2.25,0.45) circle (0.95);
		\filldraw[black] (2.25,0.45) circle (0.1) node[font=\Large,left=3pt] {$\xi_1$};
		\filldraw[black] (5.55,1.70) circle (0.1) node[font=\Large,right=3pt] {$\xi_2$};
		\node[font=\large, text=gray!90] at (3.6,-0.6) {$B_\delta(\xi_1)$};
		\draw[black, line width=2.8pt, postaction={decorate}, decoration={markings, mark=at position 0.6 with {\arrow[black,line width=2pt, scale=1.8]{>}}}] (1.5,-0.7) .. controls (2.1,0.5) and (4.35,1.5) .. (5.55,1.70);
	\end{tikzpicture}

    \caption{Left: a heteroclinic connection $[\xi_1 \to \xi_2]$ from $\xi_1$ to~$\xi_2$. Right: an excitable connection $[\xi_1 \to \xi_2]_\delta$ from $\xi_1$ to~$\xi_2$. The black trajectory does not converge to $\xi_1$ in backwards time, but instead leaves the $\delta$-neighbourhood $B_\delta(\xi_1)$ again.}

    \label{fig:het-ex}
\end{figure}

If a heteroclinic connection exists, it is not unusual for the intersection of the respective invariant manifolds to be of dimension greater than one. This implies that there are uncountably many single connecting trajectories between the equilibria, and choosing any non-empty subset leads to an invariant set that may be studied as part of a heteroclinic network.

More generally, one can consider heteroclinic connections not only between equilibria, but between arbitrary invariant sets: In line with definition~2.1 in \cite{Weinberger.2018}, we say that there is a heteroclinic connection $[S_1 \to S_2]$ between two non-empty invariant sets $S_1, S_2 \subset \R^d$, if there exists at least one point $x \in \R^d$ such that $\emptyset \ne \alpha(x) \subset S_1$ and $\emptyset \ne \omega(x)\subset S_2$, where $\alpha(x)$ and $\omega(x)$ are the $\alpha$- and $\omega$-limit sets of $x$, respectively.

In this paper, we focus on the case where the sets $S_1$ and $S_2$ are heteroclinic cycles between equilibria. This leads to the concept of \emph{depth of a heteroclinic connection}, see Definition~2.22 in~\cite{AshFie1999}: such a connection $[S_1 \to S_2]$ is of depth two (or more) if there is some $x$ in this connection so that $\alpha(x)$ and/or $\omega(x)$ contains not only equilibria. In contrast to that, a heteroclinic connection between equilibria is of depth one, since all $\alpha$- and $\omega$-limits are equilibria. 

\begin{example}
A depth two heteroclinic connection from a cycle $C$ to an equilibrium $\xi$ is sketched in~\Cref{fig:depth2}: the blue trajectory converges in backwards time to the heteroclinic cycle $C=[\xi_1 \to \xi_2 \to \xi_3]$, which is therefore its $\alpha$-limit set. In forward time, the blue trajectory converges to the single equilibrium $\xi$, thus, its $\omega$-limit set is $\{\xi \}$. This behavior can for example be observed in the well-known cubic Guckenheimer-Holmes system~\cite{Guckenheimer_Holmes_1988}.
\end{example}

\begin{figure}
    \centering
    \includegraphics[width=0.6\linewidth,trim={45mm 60mm 36mm 15mm},clip]{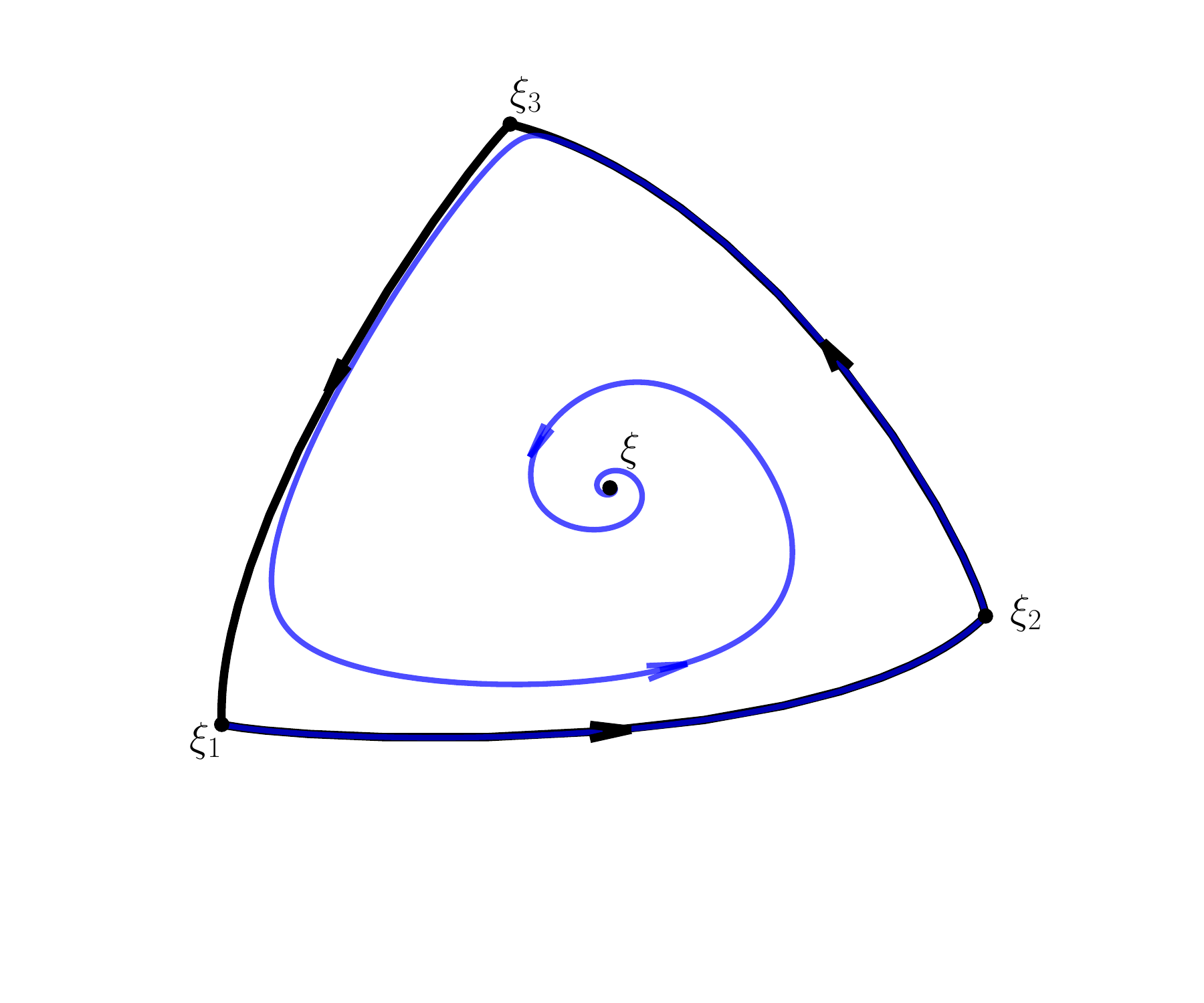}
    \caption{Heteroclinic connection of depth two from the cycle $[\xi_1 \to \xi_2 \to \xi_3]$ to~the equilibrium $\xi$; can also be interpreted as three excitable connections with threshold zero from each equilibrium $\xi_1, \xi_2, \xi_3$ to $\xi$.}
    \label{fig:depth2}
\end{figure}

Strictly speaking, if $S_1$ and $S_2$ in our definition above are heteroclinic cycles, a connection $[S_1 \to S_2]$  between them does not have to be of higher depth: If there is a depth one connection from an equilibrium in $S_1$ to another equilibrium in $S_2$ this will satisfy our definition for the existence of a connection $[S_1 \to S_2]$. A stronger requirement would be to ask for $\alpha(x) = S_1$ and $\omega(x)=S_2$ above, but that is too restrictive for our purposes when $S_1$ and $S_2$ are heteroclinic networks: A connection between two networks might typically connect only parts of the networks on both ends, e.g.\ subcycles. We point out this subtlety to alert the reader and therefore will provide additional information beyond the mere existence of a connection $[S_1 \to S_2]$ whenever appropriate.

Another, slightly weaker concept for connections in phase space is of interest, especially when heteroclinic connections are not present. Denote by $B_\delta(\cdot)$ the open neighborhood of size $\delta>0$ around a point or set in $\R^d$. An \emph{excitable connection for amplitude $\delta >0$} exists between the equilibria $\xi_1$ and $\xi_2$ if
\[B_\delta(\xi_1) \cap W^s(\xi_2) \neq \emptyset,\]
i.e., if the stable manifold of $\xi_2$ comes $\delta$-close to (but does not necessarily converge to) $\xi_1$ in backwards time, see~\cite{Ashwin.2016} and the right panel in \Cref{fig:het-ex} for an illustration. This allows for cyclic connection structures involving sinks, which is not possible for heteroclinic cycles, where every equilibrium must have at least one incoming and one outgoing direction, and thus non-trivial stable and unstable manifolds. For an excitable connection from $\xi_1$ to $\xi_2$ the threshold $\delta_{\textnormal{th}}(\xi_1,\xi_2)$ is defined as
\[\delta_{\textnormal{th}}(\xi_1,\xi_2):= \inf\{ \delta>0 \mid B_\delta(\xi_1) \cap W^s(\xi_2) \neq \emptyset \}.\]
It provides a measure for the distance between $\xi_1$ and the stable manifold of $\xi_2$. A heteroclinic connection is a special case of an excitable connection with zero threshold, because $W^u(\xi_1) \cap W^s(\xi_2) \neq \emptyset$ implies $B_\delta(\xi_1) \cap W^s(\xi_2) \neq \emptyset$ for all $\delta>0$. 

\begin{example}
As already pointed out in~\cite{Ashwin.2016}, the existence of a depth two heteroclinic connection from a cycle $C$ to an equilibrium $\xi$ implies that there are zero threshold excitable connections from each equilibrium in $C$ to $\xi$. This provides another way of looking at~\Cref{fig:depth2}: the blue trajectory belongs to $W^s(\xi)$, the stable manifold of $\xi$. In backwards time it comes arbitrarily close to each equilibrium $\xi_j$, $j \in \{1,2,3\}$, but does not belong to the unstable manifolds $W^u(\xi_j)$, bcause it converges to the full cycle $C=[\xi_1 \to \xi_2 \to \xi_3]$. So for any $\delta>0$ the blue trajectory intersects $B_\delta(\xi_j) \cap W^s(\xi)$, which is therefore non-empty. This establishes the existence of an excitable connection $[\xi_j \to \xi]_\delta$ of amplitude $\delta>0$ from each $\xi_j$ to $\xi$, even though there is no heteroclinic connection $[\xi_j \to \xi]$. Since this holds for all $\delta>0$, the excitable connection has zero threshold.
\end{example}

Excitable connections have often been studied in the presence of noise: Disturbances exceeding the threshold may push a trajectory out of the basin of attraction of $\xi_1$ so that it follows $W^s(\xi_2)$ and approaches $\xi_2$. However, we do not consider noisy systems in this paper.

We extend the notion of an excitable connection to general invariant sets, just as for heteroclinic connections above:
\begin{definition}
We say that there is an \emph{excitable connection for amplitude $\delta>0$ between two non-empty invariant sets $S_1,S_2 \subset \R^d$} if there is $x \in B_\delta(S_1)$ with $\emptyset \ne \omega(x) \subset S_2$. This is denoted as $[S_1 \to S_2]_\delta$.
\end{definition}
\noindent
The existence of an excitable connection from $S_1$ to $S_2$ means that there is a trajectory that converges to $S_2$ in forward time and comes $\delta$-close to $S_1$ at some point in the past, but typically not for $t \to -\infty$. 

As for a heteroclinic connection between two sets one may now ask whether the relevant $\omega$-limit set consists of a single equilibrium or whether it contains more points, e.g.\ a heteroclinic cycle or network. This could prompt a notion of \emph{depth of an excitable connection}. However, a formal definition seems rather technical and, for our analysis, yields no additional insight, which is why we do not take this concept further here.

In this paper, we construct vector fields which exhibit excitable connections with zero threshold between heteroclinic networks. These connections are not heteroclinic at all, i.e.,\ neither between equilibria nor between cycles/networks, because their $\alpha$-limits are empty. Since our excitable connections have zero threshold, they affect the dynamics in a way that is very similar to a heteroclinic connection in forward time, but not in backwards time. We elaborate on this later on when we explain and discuss our construction method.

%%%
\subsection{Design---The simplex method}
Let $\G$ be a directed graph on a finite set of vertices $\{v_1,\dotsc,v_n\}$. A directed edge from $v_k$ to $v_i$ is represented by $(v_k,v_i)$. Further, denote by $(A_{ki})_{k,i=1}^n$ the adjacency matrix of $\G$, i.e., its entries are $A_{ki}=1$ if there is an edge from $v_k$ to $v_i$ in $\G$, and $A_{ki}=0$ otherwise. We are interested in constructing a dynamical system~\eqref{sys1} with connections corresponding to the structure of~$\G$:
\begin{definition}
	\label{def:realization_standard}
    We say that system~\eqref{sys1} \emph{realizes $\G$ as an excitable network for amplitude $\delta>0$}, if
    \begin{enumerate}[(i)]
        \item for every vertex $v_i$ in $\G$ there is an invariant set $S_i\subset\R^d$, such that $S_i\cap S_k=\emptyset$ if $v_i\ne v_k$;
        \item there is an excitable connection $[S_k\to S_i]_{\delta}$ if and only if there is an edge $(v_k,v_i)$ in $\G$.
    \end{enumerate}
\end{definition}
\begin{remark}
    This is a straightforward generalization of the definition of a realization of a digraph as an excitable network on equilibria given in \cite{Ashwin.2016}.
\end{remark}
Several methods exist to systematically construct a vector field $f$ which realizes a given digraph $\G$, see~\cite{AshPos2013, Ashwin.2016, Field_2015, Field_2017}. In particular, any realization as a heteroclinic network is also a realization as an excitable network with any amplitude $\delta>0$ (i.e., all connections have threshold $0$).
Depending on the realization method, there are some (mild) restrictions on the digraph, and additional equilibria may be created as a byproduct.

In this subsection, we briefly review the simplex method from~\cite{AshPos2013}, which generates a heteroclinic network for a given digraph $\G$. It requires $\G$ to have no $1$- and $2$-cycles, and its key features are as follows: For each vertex in $\G$, a coordinate dimension is created, and thus, the required space dimension is equal to the number $n$ of vertices in $\G$. All equilibria belonging to the heteroclinic network lie on coordinate axes, and all coordinate axes are dynamically invariant. Coordinate planes contain the desired heteroclinic connections prescribed by $\G$, even though connections may also exist in higher-dimensional subspaces.

The vector field is polynomial and of the following form:
\begin{equation}
    \label{eq:simplex}
    \dot x_i = x_i \left(1-\|x\|^2+\sum_{k=1}^n a_{ki}x_{k}^2 \right),
\end{equation}
where $x=(x_1,\dotsc,x_n)^T\in\R^n$. If the coefficients $a_{ki}$ are chosen as
\begin{equation}
    \label{eq:simplex_coeff}
    a_{ki}>0 \text{ if } A_{ki}=1,\; a_{ki}<0 \text{ if } k \neq i \text{ and } A_{ki}=0, \text{ and } a_{ii}=0,
\end{equation}
then \eqref{eq:simplex} realizes $\G$ as a heteroclinic structure, see~\cite[Proposition~1]{AshPos2013}. The reason why this works can be summed up as follows: For every $i$ there is an equilibrium $\xi_i$ on the $x_i$-axis at $x_i=1$. Suppose there is an edge $(v_1,v_2)$ in $\G$. Then $(v_2,v_1)$ is not an edge in $\G$, as there are no $1$- or $2$-cycles. Thus, $a_{12}>0$ and $a_{21}<0$ are chosen. It follows from a direct computation that there are no equilibria in the $x_1$-$x_2$-plane except for $\xi_1$ and $\xi_2$. Moreover, the eigenvalues of the linearization of the right hand side in~\eqref{eq:simplex} are $-2$ and $a_{12}>0$ at $\xi_1$, and $-2$ and $a_{21}<0$ at $\xi_2$. Thus, there is expansion near $\xi_1$ in the direction of $\xi_2$, and $\xi_2$ is a sink in the $x_1$-$x_2$-plane. This establishes the heteroclinic connection $[\xi_1 \to \xi_2]$.

In particular, the heteroclinic structure created with this method is a realization of $\G$ in the sense of \Cref{def:realization_standard} for any amplitude $\delta>0$. Note that in system~\eqref{eq:simplex}  all coordinate axes (and thus all coordinate subspaces) are dynamically invariant, and the equation has $\Z_2^n$ symmetry, where $\Z_2$ acts by reflection across each coordinate hyperplane.

We point out that most other construction methods also realize edges in the graph as saddle-sink connections in invariant subspaces. This can be done with different types of vector fields, e.g.~involving trigonometric and hyperbolic functions in the case of the cylinder method~\cite{AshPos2013}.

%%%%
\section{Two level hierarchy---The simplex-simplex method}
\label{sec:main}

In this section, we introduce a new construction method to generate dynamical systems that exhibit heteroclinic/excitable structures with two hierarchical levels. 
Consider directed graphs $\G_1,\dotsc,\G_N$ such that $\G_j$ connects $n_j$ vertices $\{v^j_1,\dotsc,v^j_{n_j}\}$ for each $j=1,\dotsc,N$. These take the role of different connection structures on the lower hierarchical level. Further, let $\GG$ be a directed graph on $N$ vertices $\{V_1,\dotsc,V_N\}$, representing the superstructure between the individual connection structures given through the $\G_j$. As above, denote by $(A^j_{ki})_{k,i=1}^{n_j}$ and $(A_{ki})_{k,i=1}^N$ the adjacency matrices of $\G_j$ and $\GG$, respectively.

\begin{definition}
	\label{def:realization}
	We say that system~\eqref{sys1} \emph{realizes a collection of digraphs $\G_1,\dotsc,\G_N$ and $\GG$ as an excitable hierarchical network $\Net$} for amplitude $\delta>0$, if
	\begin{enumerate}[(i)]
		\item every $\G_j$ is realized as an excitable network $\Net_j$ for amplitude $\delta$,
		\item $\Net$ is a realization of $\GG$ for amplitude $\delta$ on the invariant sets $\{\Net_j\}_{j=1}^N$.
	\end{enumerate}
\end{definition}
\begin{remark}
    A hierarchical collection of digraphs $\G_1,\dotsc,\G_N$ and $\GG$ as above is not simply a larger digraph, because the edges in $\GG$ are between the digraphs $\G_j$ and not between nodes within them. Rather, it can be seen as a \emph{higher order} interaction structure, such as a directed hypergraph \cite{Ausiello.2017}. More precisely, it can be represented by generalizations of hypergraphs which allow for connections between edges, such as an unweighted $2$-depth pangraph \cite{Iskrzynski.2025}. In this interpretation, \Cref{def:realization} describes the realization of (specific instances of) such higher order structures as a heteroclinic or excitable connection structure.
\end{remark}

%%%
\subsection{Construction}
Our construction method realizes the connection structures on both hierarchical levels via the simplex method.
Consider the following system:
\begin{subequations}
\label{eq:simplex_simplex}
	\begin{align}
		\label{eq:simplex_simplex_a}
		\dot X_j &= X_j \left(1-\|X\|^2+\sum_{k=1}^N a_{kj}X_{k}^2 \right) \\[5pt]
		\label{eq:simplex_simplex_b}
		\dot x^j_i &= x^j_i \left(\left( 1 - \|x^j\|^2 + \sum_{k=1}^{n_j} \alpha^j_{ki}(x^j_k)^2 \right) b^j_\varepsilon(X) - (1-b^j_\varepsilon(X)) \right)
	\end{align}
\end{subequations}
Here, $X=(X_1,\dotsc,X_N)^T\in\R^N$ and its dynamics drives the heteroclinic super\-structure---it is decoupled from the dynamics of the remaining variables.
For suitable parameter values, it is the standard simplex realization of $\GG$.
Furthermore, $x^j=(x^j_1,\dotsc,x^j_{n_j})\in\R^{n_j}$ for $j=1,\dotsc,N$ are helper variables driving heteroclinic dynamics within the heteroclinic substructures $\G_1,\dotsc,\G_N$ with parameters chosen according to the simplex method. 

The functions $b^j_\varepsilon\colon\R^N\to\R$ are modifications of the bump or transition function $b_\varepsilon\colon\R\to\R$ given by
\begin{equation}
    \label{eq:bump}
    b_\varepsilon(z) = \begin{cases}
    	1, &z\le0,\\
        1-\frac{\exp\left(-\frac{\varepsilon}{z}\right)}{\exp\left(-\frac{\varepsilon}{z}\right)+\exp\left(-\frac{\varepsilon}{\varepsilon-z}\right)} , &0 < z <\varepsilon \\
        0, &\text{otherwise}
    \end{cases} 
\end{equation}
defined through
\[ b^j_\varepsilon(X) = b_\varepsilon(\|X-\Xi^j\|^2), \]
where
\[ \Xi^j = (0,\dotsc,0,\underset{\substack{\uparrow\\j\text{-th}}}{1}, 0,\dotsc,0)^T \]
are the equilibria of the superstructure realization \eqref{eq:simplex_simplex_a}.

Heuristically, system~\eqref{eq:simplex_simplex} is the coupling of multiple simplex systems in coordinates $x^j$, given in \eqref{eq:simplex_simplex_b}, while \eqref{eq:simplex_simplex_a} guarantees the connection via a simplex-like driving superstructure. The subsystems that realize the lower hierarchical level are further adapted as a convex combination via a bump function. It causes the dynamics of every system on the lower hierarchical level to be a realization of the corresponding digraph only near one of the equilibria of the superstructure and to decay to zero everywhere else (to guarantee stability). The role of the transition function can further be inferred from the sketch of its graph in \Cref{fig:bump}.
As there is no feedback, the system of the superstructure can be seen as a driver for the full dynamics.

\begin{figure}
    \centering
    \includegraphics[width=0.8\linewidth]{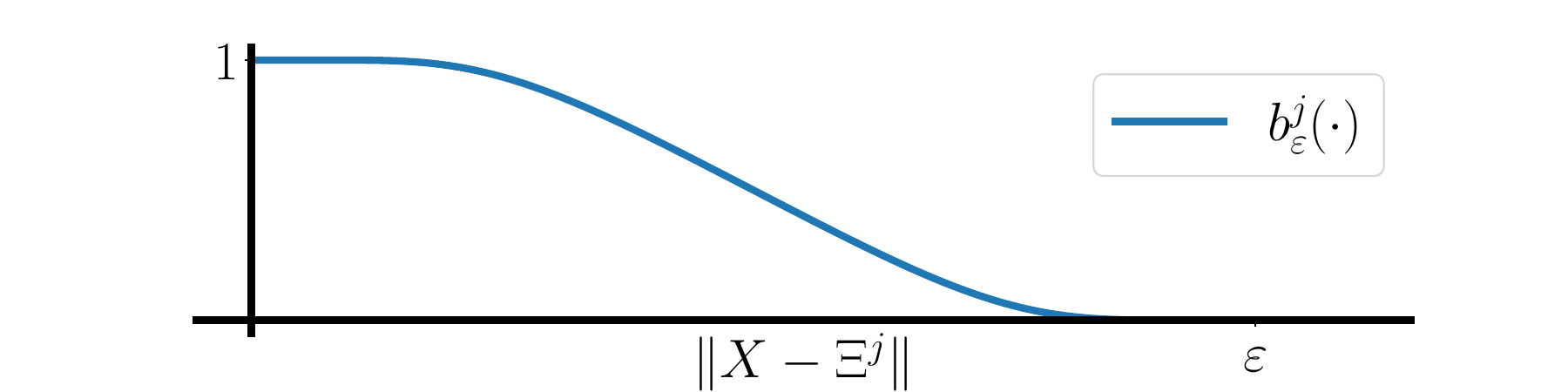}
    \caption{Sketch of the transition function $b_\varepsilon^j$ used in \eqref{eq:simplex_simplex}.}
    \label{fig:bump}
\end{figure}

%%%
\subsection{The realization result}
We formalize the heuristical description of the dynamics of \eqref{eq:simplex_simplex} as follows:
\begin{theorem}
    \label{thm:simplex_simplex}
    Let $\G_1,\dotsc,\G_N$ and $\GG$ be a collection of directed graphs as above, and suppose that they contain no $1$-cycles and no $2$-cycles. Choose the coefficients $a_{kj}$ and $\alpha^j_{ki}$ as in \eqref{eq:simplex_coeff}, according to the adjacency matrices of $\GG$ and of $\G_1,\dotsc,\G_N$, respectively, and let $\varepsilon<\frac{\sqrt{2}}{2}$. Then system~\eqref{eq:simplex_simplex} realizes this collection of digraphs as an excitable hierarchical network for any amplitude $\delta>0$. In particular, $\G_1,\dotsc,\G_N$ are realized as heteroclinic networks, while the realization of $\GG$ contains excitable connections that are not heteroclinic.
\end{theorem}
\begin{proof}
    The proof follows from a hierarchical application of Proposition 1 in \cite{AshPos2013}, which gives the reasoning behind the simplex method. For this discussion, it is convenient to denote the phase space of \eqref{eq:simplex_simplex} as
    \[ \X = \X_s \oplus \X^1\oplus\dotsb\oplus\X^N, \]
    where $\X_s$ is spanned by the coordinates $X_1,\dotsc,X_N$ and $\X^j$ for $j=1,\dotsc,N$ is spanned by the coordinates $x^j_1,\dotsc,x^j_{n_j}$. This means that $b^j_\varepsilon\colon\X_s\to\R$.

    System \eqref{eq:simplex_simplex_a} is decoupled from \eqref{eq:simplex_simplex_b} and describes the restriction of the full dynamics to $\X_s$. The system and the chosen parameters $a_{jk}$ are identical to the system discussed in Proposition~1 of \cite{AshPos2013}, which therefore guarantees that \eqref{eq:simplex_simplex_a} constitutes a realization of $\GG$ as a heteroclinic network via the simplex method. In particular, every vertex $V_j$ of $\GG$ corresponds to an equilibrium $\Xi^j\in\X_s$ and there is a heteroclinic connection from $\Xi^k$ to $\Xi^j$ if and only if there is an edge $(V_k,V_j)$ in $\GG$.
    Below, these heteroclinic connections in the dynamics of subsystem \eqref{eq:simplex_simplex_a} will be used to construct excitable connections between suitable invariant sets in the full dynamics of \eqref{eq:simplex_simplex}.
    
    To analyze the dynamics of \eqref{eq:simplex_simplex_b}, we first make the observation that the subspaces $\{x^j_i=0\}\subset\X$ are dynamically invariant for any $i, j$. Hence, so are the affine subspaces
    \begin{align*}
     \mathbb{A}^j &= \big\{ (X_1,\dotsc,X_N,x^1_1,\dotsc,x^1_{n_1},\dotsc,x^N_1,\dotsc,x^N_{n_N}) \mid\\
     &\qquad X_j=1, X_k=x^k_i=0 \text{ for all } k\ne j, i=1,\dotsc,n_k \big\},
     \end{align*}
    as therein $X=\Xi^j$, which is an equilibrium for \eqref{eq:simplex_simplex_a}. In $\mathbb{A}^j$, the only (dynamically) non-trivial coordinates are $x^j_1,\dotsc,x^j_{n_j}$. Their dynamics (in $\mathbb{A}^j$) is governed by \eqref{eq:simplex_simplex_b}, which reduces to
    \[ \dot x^j_i = x^j_i \left( 1 - \|x^j\|^2 + \sum_{k=1}^{n_j} \alpha^j_{ki}(x^j_k)^2 \right) \]
    due to the fact that $X=\Xi^j$ which causes $b^j_\varepsilon(X)=1$. This system and the chosen parameters $\alpha^j_{ki}$ are identical to the system discussed in Proposition~1 of \cite{AshPos2013}, which therefore guarantees that the restricted system constitutes a realization of $\G_j$ as a heteroclinic network via the simplex method. In particular, every vertex $v^j_i$ of $\G_j$ corresponds to an equilibrium $\xi^j_i\in\mathbb{A}^j$ with $x^j_i=1$ and $x^j_k=0$ for all $k\ne i$ and there is a heteroclinic connection from $\xi^j_k$ to $\xi^j_i$ if and only if there is an edge $(v^j_k,v^j_i)$ in $\G_j$. As before, each of these heteroclinic connections in particular constitutes an excitable connection $[\xi^j_k\to \xi^j_i]_\delta$ for any $\delta>0$. Note that $\xi^j_i$ are equilibria of the full dynamics of \eqref{eq:simplex_simplex}. The heteroclinic network consisting of these equilibria and the corresponding heteroclinic connections between them is a dynamically invariant set which we denote by $\Net_j\subset\mathbb{A}^j\subset\X$.

    To complete the proof, we show that there is indeed an excitable connection $[\Net_k\to\Net_j]_\delta$ for any $\delta>0$ in the full dynamics of \eqref{eq:simplex_simplex} if and only if there is an edge $(V_k,V_j)$ in $\GG$. To that end, we first observe that the supports of the bump functions $b^j_\varepsilon$ are pairwise disjoint for different values of $j$. This is due to the position of the $\Xi^j$ on the coordinate axes and the specific choice of $\varepsilon<\frac{\sqrt{2}}{2}$.
    Now fix an arbitrary $\delta>0$ and consider a trajectory \[\XX(t)=(X_1(t),\dotsc,X_N(t),x^1_1(t),\dotsc,x^1_{n_1}(t),\dotsc,x^N_1(t),\dotsc,x^N_{n_N}(t))\] with initial condition \[\XX=(X_1,\dotsc,X_N,x^1_1,\dotsc,x^1_{n_1},\dotsc,x^N_1,\dotsc,x^N_{n_N})\] 
    in the subspace
    \[ \{X_l=x^l_i=0 \text{ for all } l\ne k,j  \text{ and all } i\} \]
    which is dynamically invariant. In particular, we choose the coordinate entries of $\XX$ such that $\XX\in B_\delta(\Net_k)$. This implies the following:
    \begin{itemize}
        \item $X_k$ is close to $1$,
        \item $X_j$ is close to $0$,
        \item $x^k_1,\dotsc,x^k_{n_k}$ are close to the heteroclinic network realized by the simplex system \eqref{eq:simplex_simplex_b} reduced to $\mathbb{A}^k$,
        \item $x^j_1,\dotsc,x^j_{n_j}$ are close to $0$.
    \end{itemize}
    Furthermore, we choose
    \begin{itemize}
        \item $x^j_1>0$ and $x^j_2,\dotsc,x^j_{n_j}=0$, 
        \item $X_k$ and $X_j$ such that $X=(X_1,\dotsc,X_N)$ lies on a connecting trajectory between $\Xi^k$ and $\Xi^j$ for the restricted dynamics \eqref{eq:simplex_simplex_a}. 
    \end{itemize}
    Then, by construction, in forward time $X(t)=(X_1(t),\dotsc,X_N(t))$ determined by \eqref{eq:simplex_simplex_a} converges to $\Xi^j$ while only $X_k, X_j$ are non-zero. In particular, there are times $t^{**}>t^*>0$ such that
    \begin{itemize}
        \item $X(t)\in B_\varepsilon(\Xi^k)$ for $t<t^*$,
        \item $X(t)\notin B_\varepsilon(\Xi^k) \cup B_\varepsilon(\Xi^j)$ for $t\in(t^*,t^{**})$, and
        \item $X(t)\in B_\varepsilon(\Xi^j)$ for $t>t^{**}.$
    \end{itemize}
    For $t\in(t^*,t^{**})$ the convex combination in \eqref{eq:simplex_simplex_b} for $l, i$ is such that all $x_i^l(t)$ decay towards $0$. As soon as $t>t^{**}$, only the $j$-th subsystem slowly changes. Due to the invariance of the coordinate subspaces and our choice of initial condition, we have
    \[x_2^j(t)=x_3^k(t)=\dotsc=x_{n_j}^j(t)=0 \]
    for all $t$, as well as
	\[ \dot x^j_1 = x^j_1 \left( (1 - (x^j_1)^2) b^j_\varepsilon(X) - (1-b^j_\varepsilon(X))\right). \] 
    The zeros of the right hand side of this equation for fixed values of $X$ are 
    \[x_1^j=0 \quad \text{and} \quad x_1^j=\pm \sqrt{2-\frac{1}{b_\varepsilon^j(X)}}.\]
    In particular, along the trajectory $\XX(t)$ through the initial condition specified above, for $t \to \infty$, we get $x_1^j(t) \to 1$ since $b^j_\varepsilon(X(t)) \to b^j_\varepsilon(\Xi^j)= 1$. Therefore, the trajectory $\XX(t)$ converges to the equilibrium where $X=\Xi^j$, $x_1^j=1$ and $x_i^l=0$ for all other $l,i$, which belongs to $\Net_j \subset \mathbb{A}^j$. This means $\omega(\XX(t))\subset\Net_j$. As $\delta>0$ was chosen arbitrarily, this finalizes the proof that there exists an excitable connection $[\Net_k\to\Net_j]_\delta$ for any $\delta>0$.

    Conversely, if there is no edge $(V_k,V_j)$ in $\GG$, then no such excitable connection $[\Net_k\to\Net_j]_\delta$ exists. The reason is that there is not even a connection from $\Xi^k$ to $\Xi^j$ in the reduced dynamical system \eqref{eq:simplex_simplex_a} because the simplex method creates dynamics where both $\Xi^k$ and $\Xi^j$ are stable within the subspace $\{X_l = 0 \mid l\ne k,j\}$ if there is no edge between $V_k$ and $V_j$ in $\GG$. This rules out the existence of solutions starting near $\Xi^k$ and converging to $\Xi^j$.

    Finally, the excitable connection $[\Net_k\to\Net_j]_\delta$ constructed above does not form a heteroclinic connection and no such heteroclinic connection can exist in the invariant subspace
    \[ \{X_l = 0 \mid l\ne k,j\}, \]
    which contains $\Net_k, \Net_j$. A trajectory $\XX(t)$ in this subspace with $\omega(\XX(t))\subset\Net_j$ requires $x^j_i\ne0$ for some $i$. If additionally $\alpha(\XX(t))\subset\Net_k$ then necessarily $\lim_{t\to-\infty}X(t)=\Xi^j$. Hence, for some $t^*$ we have that the $k$-th subsystem of \eqref{eq:simplex_simplex_b} reduces to
    \[ \dot x^j_i = -x^j_i \]
    for all $t<t^*$. In particular, $\lim_{t\to-\infty}x^j_i(t)=\pm\infty$. This means $\XX(t)$ has no accumulation points. Thus, we have $\alpha(\XX(t))=\emptyset$ and $\XX(t)$ is not a heteroclinic connection from $\Net_k$ to $\Net_j$.
\end{proof}

In the proof above, we established the existence of excitable connections with equilibria as their $\omega$-limit sets. It seems likely that, for most choices of the directed graphs, more complex connections exist in the resulting system, i.e.,\ trajectories which converge to a full subnetwork $\Net_j$ or a non-trivial heteroclinic substructure of $\Net_j$. This is linked to the stability of $\Net_j$: When $\Net_j$ is, say, fragmentarily asymptotically stable in $\mathbb{A}^j$, it is reasonable to expect a connection with a larger $\omega$-limit set than just one equilibrium.

For networks created with the simplex method, no general stability results are known: While the generated network is never asymptotically stable if it contains at least one equilibrium with a higher-dimensional ($\geq 2$) unstable manifold, Ashwin and Postlethwaite~\cite{AshPos2013} conjecture that there is a larger network which can be asymptotically stable, depending on the chosen parameters. This larger network must include the closures of all unstable manifolds and therefore, typically, involves additional equilibria. There is no general way to control the transverse stability of these equilibria and therefore the overall stability of the network: This is addressed in~\cite{AshCasLoh2020}, where an explicit example is given as well.

Our numerical results in \Cref{sec:numerics} support the conjecture that the subnetworks $\Net_j$ usually possess some stability. Deeper insight into this aspect will give further information on the exact nature of the connections created in our construction, but is beyond the scope of this paper.

%%%
\subsection{Robustness}
It is worth noting that in system~\eqref{eq:simplex_simplex} all coordinate subspaces are dynamically invariant by construction. We therefore focus on the non-negative orthant to understand the dynamics. Both the heteroclinic and excitable connections lie in invariant subspaces  and the respective equilibria in the $\omega$-limit sets are sinks within these. Therefore, all connections are robust in the sense that they persist under perturbations to system~\eqref{eq:simplex_simplex} which respect this invariance.

In the original simplex method, there is a $\Z_2$ symmetry in every coordinate, given through the reflection across the respective hyperplane, i.e., the system is equivariant with respect to changing the sign of any component in the state vector. In our system~\eqref{eq:simplex_simplex}, this symmetry is only partially preserved: The subsystem~\eqref{eq:simplex_simplex_a} is $\Z_2^N$ equivariant. The full system, however, respects sign changes of coordinate entries only in the $x^j_i$-entries. Hence, it is $\Z_2^{n_1} \times \ldots \times \Z_2^{n_N}$ equivariant. The reason for this discrepancy is that in the equations for the $x_i^j$-variables~\eqref{eq:simplex_simplex_b} the bump functions $b_\varepsilon^j$ do not respect the symmetry in $X$. A symmetric modification of the $b_\varepsilon^j$ could change this, so that one obtains full $\Z_2^N \times \Z_2^{n_1} \times \ldots \times \Z_2^{n_N}$ symmetry. However, as pointed out above, due to the invariance of the coordinate subspaces this is not necessary to ensure robustness.

%%%
\subsection{Thresholds of excitable connections}
In~\cite{Ashwin.2016}, an excitable connection for amplitude $\delta>0$ is called \emph{proper} if it has nonzero threshold. This means there is $\delta' \in (0, \delta)$ such that no excitable connection exists for amplitude $\delta'$. Since we have shown in the proof of \Cref{thm:simplex_simplex} that excitable connections $[\Net_k \to \Net_j]_\delta$ exist for any amplitude $\delta>0$ when prescribed by the digraph $\GG$, these connections have threshold zero and are thus not properly excitable.

In~\Cref{subsec:ex-het}, we mentioned an example for excitable connections with threshold zero that is also shown in~\cite{Ashwin.2016}, see~\Cref{fig:depth2}: There is a trajectory converging in backwards time to a heteroclinic cycle $[\xi_1 \to \xi_2 \to \xi_3 \to \xi_1]$ of length three, and in forward time to another equilibrium $\xi$. That is, for all $\delta>0$ it passes through every neighborhood $B_\delta(\xi_j)$, for $j \in \{1,2,3 \}$. Therefore, this trajectory constitutes an excitable connection of threshold zero between $\xi_j$ and $\xi$ for $j \in \{1,2,3 \}$. An alternative characterization of this situation is to say that there is a heteroclinic connection of depth two from the cycle to the equilibrium $\xi$: Its $\omega$-limit is $\{\xi\}$, but its $\alpha$-limit consists of the entire cycle, and thus contains not only equilibria, but also the heteroclinic connections between them.

In our construction, we showed that the $\alpha$-limit sets of the excitable connections are empty since all respective trajectories have unbounded components in backwards time. Therefore, they are not heteroclinic of any depth. This distinguishes our connections from those in \cite{Ashwin.2016}.

%%%
\subsection{Modification of time-scales}
\label{subsec:modification}
When it comes to practically observing trajectories that display the hierarchical transition structure according to a given collection of digraphs---e.g. in numerical simulations---, one has to mitigate effects on different and partially conflicting timescales, which depend on the stability properties of the network. Heteroclinic and excitable connections have the property of \emph{exponentially decaying remaining times}. That is, a trajectory initialized near an attracting heteroclinic/excitable connection structure spends longer and longer times near each individual equilibrium or invariant subset. The precise times and, thus, the sequence of transitions depends on the initial condition and the parameter values $a_{kj}, \alpha^j_{ki}$ in system \eqref{eq:simplex_simplex}. For some choices, it may happen that certain transitions dictated by the collection of digraphs are not observed, even though the connections exist in phase space. For example, the transitions of the superstructure may happen too fast for any of the lower hierarchical transitions to take place.

To mitigate such effects, we add additional parameters into the system, which can be used to speed up or slow down certain parts of its dynamics. Consider the adapted system
\begin{subequations}
	\label{eq:simplex_simplex_adapted}
	\begin{align}
		\label{eq:simplex_simplex_adapted_a}
		\dot X_j	&= \Phi\cdot X_j \left(1-\|X\|^2+\sum_{k=1}^N a_{kj}X_{k}^2 \right) \\[5pt]
		\label{eq:simplex_simplex_adapted_b}
		\dot x^j_i	&= x^j_i \left(\Psi\cdot\left( 1 - \|x^j\|^2 + \sum_{k=1}^{n_j} \alpha^j_{ki}(x^j_k)^2 \right) b^j_\varepsilon(X) - \Omega\cdot(1-b^j_\varepsilon(X)) \right).
	\end{align}
\end{subequations}
This system is the same as the original realization \eqref{eq:simplex_simplex} except for the parameters $\Phi, \Psi, \Omega >0$---more precisely, for $\Phi=\Psi=\Omega=1$, both systems are identical. The additional parameters play the following roles:
\begin{itemize}
    \item $\Phi$ adapts the velocity of the superstructure dynamics,
    \item $\Psi$ adapts the velocity of the dynamics of the substructure dynamics, when they are active according to the superstructure dynamics,
    \item $\Omega$ adapts the velocity of the decay to $0$ of the substructure variables, when the respective substructure is not active according to the superstructure dynamics.
\end{itemize}
In comparison to system \eqref{eq:simplex_simplex}, a parameter $\Lambda\in\{\Phi,\Psi,\Omega\}$ thus speeds up its respective component of the overall dynamics if $\Lambda>1$ and slows it down if $\Lambda<1$. For additional fine-tuning further parameters could be introduced, for example, by making $\Psi$ and $\Omega$ dependent on $j$, i.e., allow them to differ across the individual substructures.

Finally, we point out that none of these changes alter the validity of our existence result in~\Cref{thm:simplex_simplex}. In our proof, the existence of the invariant sets and the equilibria depends on the zeros of functions which are now multiplied by one of the parameters $\Lambda\in\{\Phi,\Psi,\Omega\}$. The existence of the connections further depends on signs of eigenvalues of these functions. None of these aspects change, when the entire function is multiplied with a positive parameter. Thus, the adapted system~\eqref{eq:simplex_simplex_adapted} realizes the same hierarchical collection of digraphs as \eqref{eq:simplex_simplex}, when the parameters $a_{kj}, \alpha^j_{ki}, \varepsilon$ are chosen as in~\Cref{thm:simplex_simplex}.

%%%%
\section{Examples and simulations}
\label{sec:numerics}
In this section, we use the simplex-simplex method to realize two collections of digraphs as excitable networks. In both cases, we choose suitable parameters and simulate the trajectory of an exemplary initial condition near the excitable network. These follow the excitable networks and reveal the desired topological structure, indicating that our construction yields invariant objects which bear some stability. This aspect is still an open challenge for simplex method realizations, and we do not investigate it further here.

The structures we have chosen are simple but non-trivial examples of hierarchical collections of digraphs, which can be realized by our method. Their building blocks are digraphs consisting of three or four vertices. In both collections, a single non-cyclic digraph appears, see~\Cref{fig:kirk-silber}. It corresponds to the well-studied Kirk-Silber network~\cite{Kirk_1994}. We have chosen this digraph as an element of both constructions because it is a basic example of a digraph containing two cycles. It has been shown that the simplex realization of this digraph as a heteroclinic network exhibits a weak form of \emph{switching dynamics}: For suitable parameter values, there are trajectories which follow the upper cycle for a certain amount of time before making a switch to the lower cycle which they then follow for all future times. This has been observed in~\cite{Kirk_1994} already and investigated in more detail in~\cite{Lohse_Diss}. We will see below that our method correctly reproduces this behavior.
The precise details of the respective constructions for both examples are also discussed below.

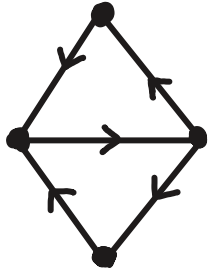
\begin{figure}
    \centering
    \resizebox{.25\linewidth}{!}{
    	\input{img/ks-network}
    }
    \caption{Digraph on four vertices corresponding to the Kirk-Silber network.}
    \label{fig:kirk-silber}
\end{figure}

%%%
\subsection{Switching substructure}
\label{subsec:cyclic-super}
\begin{figure}
    \centering
	\resizebox{.5\linewidth}{!}{
		\input{img/cyclic-superstructure}
	}
    \caption{Sketch of a hierarchical collection of digraphs. The superstructure and two of the substructures correspond to $3$-cycles, while the third substructure corresponds to the Kirk-Silber network.}
    \label{fig:cyclic-super}
\end{figure}
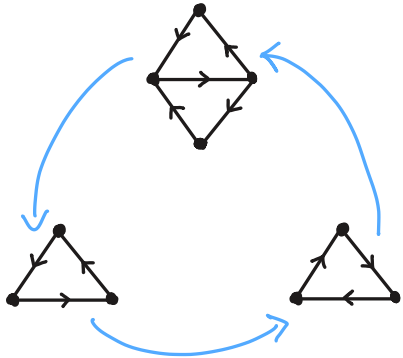
The first hierarchical structure we consider is sketched in \Cref{fig:cyclic-super}: The superstructure digraph $\GG$ is a directed $3$-cycle, while the three substructures are two $3$-cycles $\G_1, \G_2$ and a Kirk-Silber network $\G_3$. Although there are no vertex labels in the sketch, it is indicated that the edges of both substructure $3$-cycles point in opposite directions. By choosing labels for all vertices for all four digraphs, we obtain adjacency matrices
\begin{equation*}
    A = \begin{pmatrix}
        0 & 1 & 0 \\
        0 & 0 & 1 \\
        1 & 0 & 0
    \end{pmatrix},\ 
    A^1 = \begin{pmatrix}
        0 & 1 & 0 \\
        0 & 0 & 1 \\
        1 & 0 & 0
    \end{pmatrix},\
    A^2 = \begin{pmatrix}
        0 & 0 & 1 \\
        1 & 0 & 0 \\
        0 & 1 & 0
    \end{pmatrix},\
    A^3 = \begin{pmatrix}
        0 & 1 & 0 & 0 \\
        0 & 0 & 1 & 1 \\
        1 & 0 & 0 & 0 \\
        1 & 0 & 0 & 0
    \end{pmatrix}.
\end{equation*}

For the realization \eqref{eq:simplex_simplex} (or \eqref{eq:simplex_simplex_adapted}), we choose system parameters in accordance with \eqref{eq:simplex_coeff}. These can conveniently be represented in matrix form as
\begin{align*}
	(a_{kj}) &= \begin{pmatrix}
		0 & 1 & -1.5 \\
		-1.5 & 0 & 1 \\
		1 & -1.5 & 0
	\end{pmatrix}, 
	&(\alpha^1_{ki}) &= \begin{pmatrix}
		0 & 1 & -1.1 \\
		-1.1 & 0 & 1 \\
		1 & -1.1 & 0
	\end{pmatrix} \\
	(\alpha^2_{ki}) &= \begin{pmatrix}
		0 & -1.1 & 1 \\
		1 & 0 & -1.1 \\
		-1.1 & 1 & 0
	\end{pmatrix}, 
	&(\alpha^3_{ki}) &= \begin{pmatrix}
		0 & 1 & -1.5 & -1.5 \\
		-1.5 & 0 & 0.5 & 2 \\
		1 & -1.5 & 0 & -1.5 \\
		1 & -1.5 & -1.5 & 0
	\end{pmatrix}
\end{align*}
Furthermore, we set
\[ \varepsilon=0.2<\frac{\sqrt{2}}{2}. \]
This choice of parameters fully determines system \eqref{eq:simplex_simplex} and guarantees that the hierarchical interaction structure in \Cref{fig:cyclic-super} is realized as an excitable network.

To simulate the system and observe the excitable interaction structure, we make use of the adaptation outlined in \Cref{subsec:modification}. In particular, we choose parameters
\[ \Phi = 0.1, \quad \Psi = 200, \quad \Omega = 0.05. \]
With these additional parameters, the adapted system \eqref{eq:simplex_simplex_adapted} is fully determined.

\begin{landscape}
\begin{figure}
    \centering
    \includegraphics[width=\linewidth, trim={5cm .5cm 4cm 1cm}, clip]{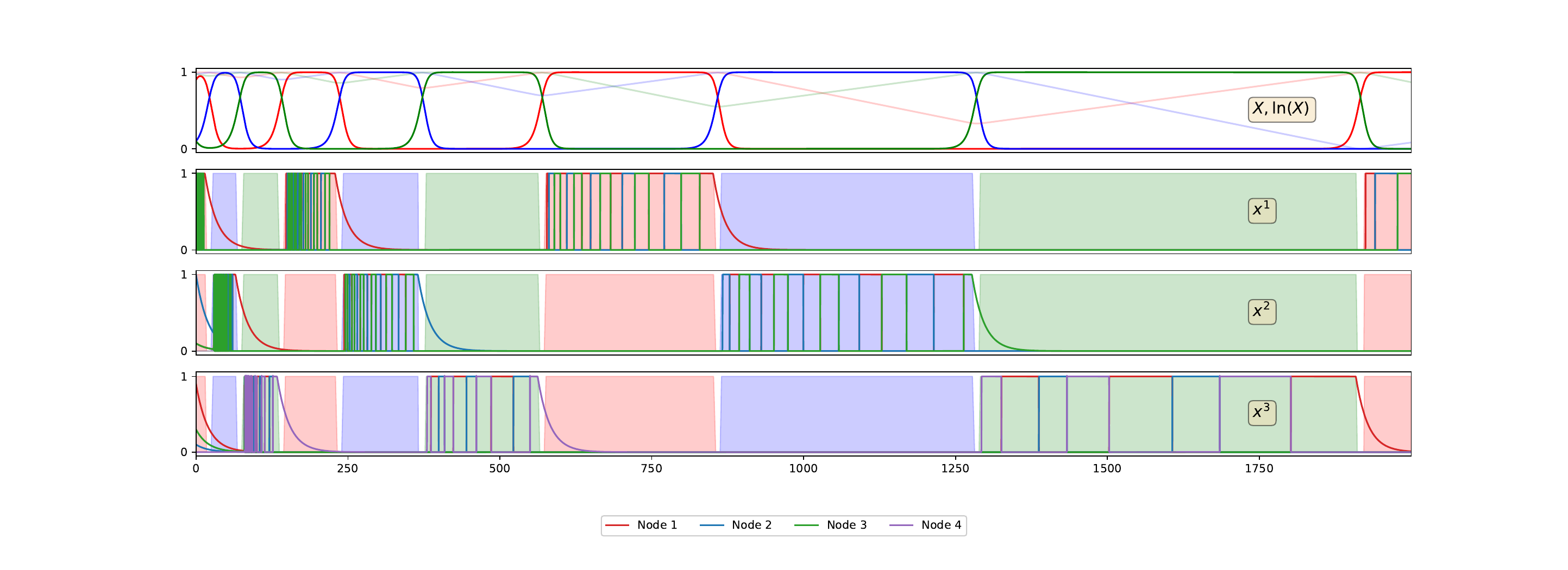}
    \caption{Timeseries of the realization of the digraph in \Cref{fig:cyclic-super}. The top panel shows the dynamics of the superstructure variables $X$. The three lower panels show the dynamics of the three substructure systems with coordinates $x^1, x^2, x^3$, respectively. The colored lines in all four panels correspond to the $3$, respectively $4$ components of the vectors and are color coded according to the legend.
    	The colored overlay in the lower three panels indicates which of the subsystems is active, i.e., which equilibrium the $X$ component is close to. A full discussion including parameter values and interpretation can be found in \Cref{subsec:cyclic-super}.}
    \label{fig:cyclic-super-timeseries}
\end{figure}
\end{landscape}
\paragraph{Numerical simulation.}
We simulate system \eqref{eq:simplex_simplex_adapted} determined by the choices above for the initial condition
\begin{equation*}
    X = \begin{pmatrix}
        0.9 \\ 0.1 \\ 0.1
    \end{pmatrix}, \quad x^1 = \begin{pmatrix}
        0.999 \\0.1 \\ 0.1
    \end{pmatrix}, \quad x^2=\begin{pmatrix}
        0.1 \\ 0.999 \\0.1 
    \end{pmatrix}, \quad x^3=\begin{pmatrix}
        0.9 \\ 0.1 \\ 0.3 \\ 0.000001
    \end{pmatrix}.
\end{equation*}
The trajectory of this initial condition is computed until $t=2000$. A timeseries is shown in \Cref{fig:cyclic-super-timeseries}. The top panel contains the dynamics of the superstructure variables $X$ (more precisely, each colored line displays the evolution of one of its components, color coded as in the legend.
We observe cyclic transitions
\begin{equation*}
    \begin{pmatrix}
    	X_1 \\ X_2 \\ X_3
    \end{pmatrix} \quad \approx \quad \begin{pmatrix}
        1 \\ 0 \\ 0
    \end{pmatrix} \to \begin{pmatrix}
        0 \\ 1 \\ 0
    \end{pmatrix} \to \begin{pmatrix}
        0 \\0 \\ 1
    \end{pmatrix} \to 
    \begin{pmatrix}
        1 \\ 0 \\ 0
    \end{pmatrix} \to \cdots
\end{equation*}
of the subsystem corresponding to the superstructure $\GG$. In fact, the trajectory does not reach any of these points, but gets closer and closer, spending increasingly longer times in their vicinity, as indicated by the shaded, logarithmically transformed timeseries.

The lower three panels of \Cref{fig:cyclic-super-timeseries} show the timeseries of the three subsystems in variables $x^1, x^2, x^3$ (which are $3$-, $3$-, and $4$-dimensional, respectively) using the color code in the legend.
Here, we see that each of the subsystems is only ``active'' when the $X$ trajectory is near the corresponding equilibrium, otherwise all components in this subsystem decay to zero. \Cref{fig:cyclic-super-timeseries-zoom} shows zooms of the substructure timeseries in exemplary timeranges. In panel~(a), subsystem $1$ follows the cycle
\begin{equation*}
    \begin{pmatrix}
    	x^1_1 \\ x^1_2 \\ x^1_3
    \end{pmatrix} \quad \approx \quad \begin{pmatrix}
        1 \\ 0 \\ 0
    \end{pmatrix} \to \begin{pmatrix}
        0 \\ 1 \\ 0
    \end{pmatrix} \to \begin{pmatrix}
        0 \\0 \\ 1
    \end{pmatrix} \to 
    \begin{pmatrix}
        1 \\ 0 \\ 0
    \end{pmatrix} \to \dotsb
\end{equation*}
In panel~(b), subsystem $2$ follows the opposite cycle
\begin{equation*}
    \begin{pmatrix}
    	x^2_1 \\ x^2_2 \\ x^2_3
    \end{pmatrix} \quad \approx \quad \begin{pmatrix}
        1 \\ 0 \\ 0
    \end{pmatrix} \to \begin{pmatrix}
        0 \\ 0 \\ 1
    \end{pmatrix} \to \begin{pmatrix}
        0 \\1 \\ 0
    \end{pmatrix} \to 
    \begin{pmatrix}
        1 \\ 0 \\ 0
    \end{pmatrix} \to \dotsb
\end{equation*}
Finally, subsystem $3$ in panel~(c), when first active, displays the following transitions:
\begin{align*}
    \begin{pmatrix}
    	x^3_1 \\ x^3_2 \\ x^3_3 \\ x^3_4
    \end{pmatrix} \quad \approx \quad &\begin{pmatrix}
        1 \\ 0 \\ 0 \\ 0
    \end{pmatrix} \to 
    \begin{pmatrix}
        0 \\ 1 \\ 0 \\ 0
    \end{pmatrix} \to 
    \begin{pmatrix}
        0 \\ 0 \\ 1 \\ 0
    \end{pmatrix} \to 
    \begin{pmatrix}
        1 \\ 0 \\ 0 \\ 0
    \end{pmatrix} \\[5pt]
    &\to 
    \begin{pmatrix}
        0 \\ 1 \\ 0 \\ 0
    \end{pmatrix} \to 
    \begin{pmatrix}
        0 \\ 0 \\ 0 \\ 1
    \end{pmatrix} \to 
    \begin{pmatrix}
        1 \\ 0 \\ 0 \\ 0
    \end{pmatrix} \to
    \begin{pmatrix}
        0 \\ 1 \\ 0 \\ 0
    \end{pmatrix} \to 
    \begin{pmatrix}
        0 \\ 0 \\ 0 \\ 1
    \end{pmatrix} \to 
    \begin{pmatrix}
        1 \\ 0 \\ 0 \\ 0
    \end{pmatrix} \to\dotsb
\end{align*}
That is, when the third subsystem is active for the first time, it follows the upper cycle of the Kirk-Silber network once, before following the lower cycle (i.e., the one containing the equilibrium point $(0, 0, 0, 1)^T$ but not $(0,0,1,0)^T$). This transition occurs only in the first active time span. Afterwards, the $x^3$ components of the trajectory continue following the lower cycle exclusively.
Summarizing, our simulation displays the desired behavior that was proven to exist in \Cref{thm:simplex_simplex}: A driving superstructure reflecting $\GG$ and three subsystems that behave according to $\G_1,\G_2,\G_3$.

\begin{figure}
    \centering
	\begin{overpic}[width=\linewidth]{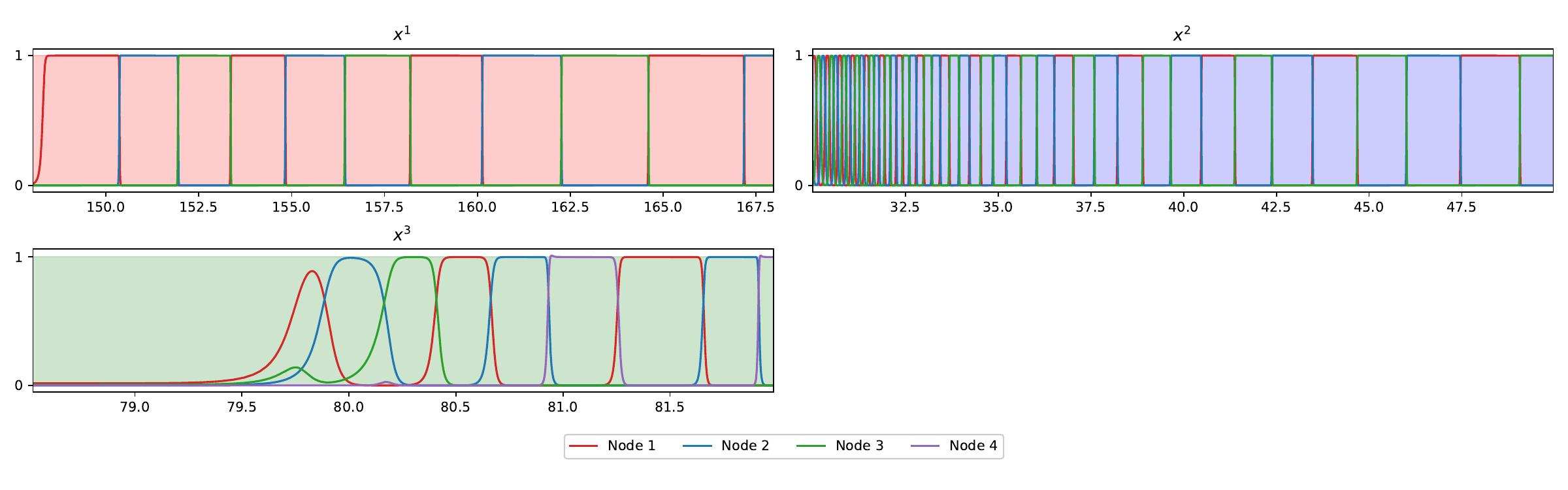}
        \put(2,29){(a)}
        \put(52,29){(b)}
        \put(2,16){(c)}
    \end{overpic}
    \caption{Zooming into substructure dynamics of \Cref{fig:cyclic-super-timeseries} for suitably chosen timeranges. (a) Subsystem $1$ displays cyclic behavior. (b) Subsystem $2$ displays cyclic behavior in the opposite direction. (c) Subsystem $3$ displays switching behavior according to the Kirk-Silber network.}
    \label{fig:cyclic-super-timeseries-zoom}
\end{figure}

%%%
\subsection{Switching superstructure}
\label{subsec:ks-super}
\begin{figure}
    \centering
	\resizebox{.6\linewidth}{!}{
		\input{img/ks-superstructure}
	}
    \caption{Sketch of a hierarchical collection of digraphs. The superstructure corresponds to the Kirk-Silber network, the substructures are $3$- and $4$-cycles in opposing directions.}
    \label{fig:ks-super}
\end{figure}
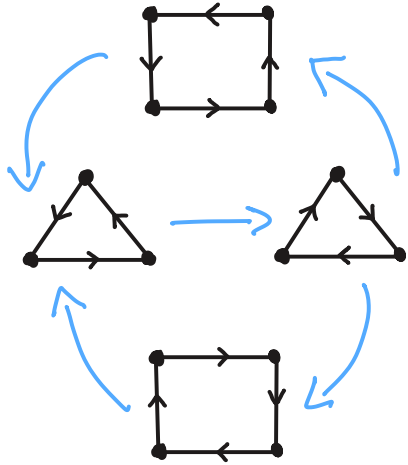
We now turn to the hierarchical structure sketched in \Cref{fig:ks-super}: The superstructure digraph $\GG$ is the Kirk-Silber network (four vertices), while the four substructures are two $3$-cycles $\G_1, \G_2$ and two $4$-cycles $\G_3, \G_4$ with connections in opposite directions, respectively. By choosing labels for the vertices in all four digraphs, we obtain adjacency matrices
\begin{align*}
    A &= \begin{pmatrix}
        0 & 1 & 0 & 0 \\
        0 & 0 & 1 & 1 \\
        1 & 0 & 0 & 0 \\
        1 & 0 & 0 & 0
    \end{pmatrix}, & & \\
    A^1 &= \begin{pmatrix}
        0 & 1 & 0 \\
        0 & 0 & 1 \\
        1 & 0 & 0
    \end{pmatrix},
    &A^2 &= \begin{pmatrix}
        0 & 0 & 1 \\
        1 & 0 & 0 \\
        0 & 1 & 0
    \end{pmatrix},\\
    A^3 &= \begin{pmatrix}
        0 & 1 & 0 & 0 \\
        0 & 0 & 1 & 0 \\
        0 & 0 & 0 & 1 \\
        1 & 0 & 0 & 0
    \end{pmatrix}
    &A^4 &= \begin{pmatrix}
        0 & 0 & 0 & 1 \\
        1 & 0 & 0 & 0 \\
        0 & 1 & 0 & 0 \\
        0 & 0 & 1 & 0
    \end{pmatrix}.
\end{align*}

\clearpage
For the realization \eqref{eq:simplex_simplex} (or \eqref{eq:simplex_simplex_adapted}), we choose system parameters in accordance with \eqref{eq:simplex_coeff}. These can conveniently be represented in matrix form as
\begin{align*}
	(a_{kj}) &= \begin{pmatrix}
		0 & 1 & -1.5 & -1.5 \\
		-1.5 & 0 & 0.5 & 2 \\
		1 & -1.5 & 0 & -1.5 \\
		1 & -1.5 & -1.5 & 0
	\end{pmatrix}, & & \\
	(\alpha^1_{ki}) &= \begin{pmatrix}
		0 & 1 & -1.1 \\
		-1.1 & 0 & 1 \\
		1 & -1.1 & 0
	\end{pmatrix}
	& (\alpha^2_{ki}) &= \begin{pmatrix}
		0 & -1.1 & 1 \\
		1 & 0 & -1.1 \\
		-1.1 & 1 & 0
	\end{pmatrix}, \\
	(\alpha^3_{ki}) &= \begin{pmatrix}
		0 & 1 & -1.01 & -1.1 \\
		-1.01 & 0 & 1 & -1.01 \\
		-1.01 & -1.01 & 0 & 1 \\
		1 & -1.01 & -1.01 & 0
	\end{pmatrix}, && \\
	(\alpha^4_{ki}) &= \begin{pmatrix}
		0 & -1.01 & -1.01 & 1 \\
		1 & 0 & -1.01 & -1.01 \\
		-1.01 & 1 & 0 & -1.01 \\
		-1.01 & -1.01 & 1 & 0
	\end{pmatrix}. &&
\end{align*}
To fully determine \eqref{eq:simplex_simplex}, we again set
\[ \varepsilon=0.2<\frac{\sqrt{2}}{2}. \]
This guarantees that the hierarchical interaction structure in \Cref{fig:ks-super} is realized as an excitable network.

As before, we choose timescale adaptation parameters
\[ \Phi = 0.1, \quad \Psi = 200, \quad \Omega = 0.05. \]
to simulate the system and observe the excitable interaction structure. This determines the adapted system \eqref{eq:simplex_simplex_adapted}.

\begin{landscape}
\begin{figure}
    \centering
    \includegraphics[width=\linewidth, trim={5cm .5cm 4cm 1cm}, clip]{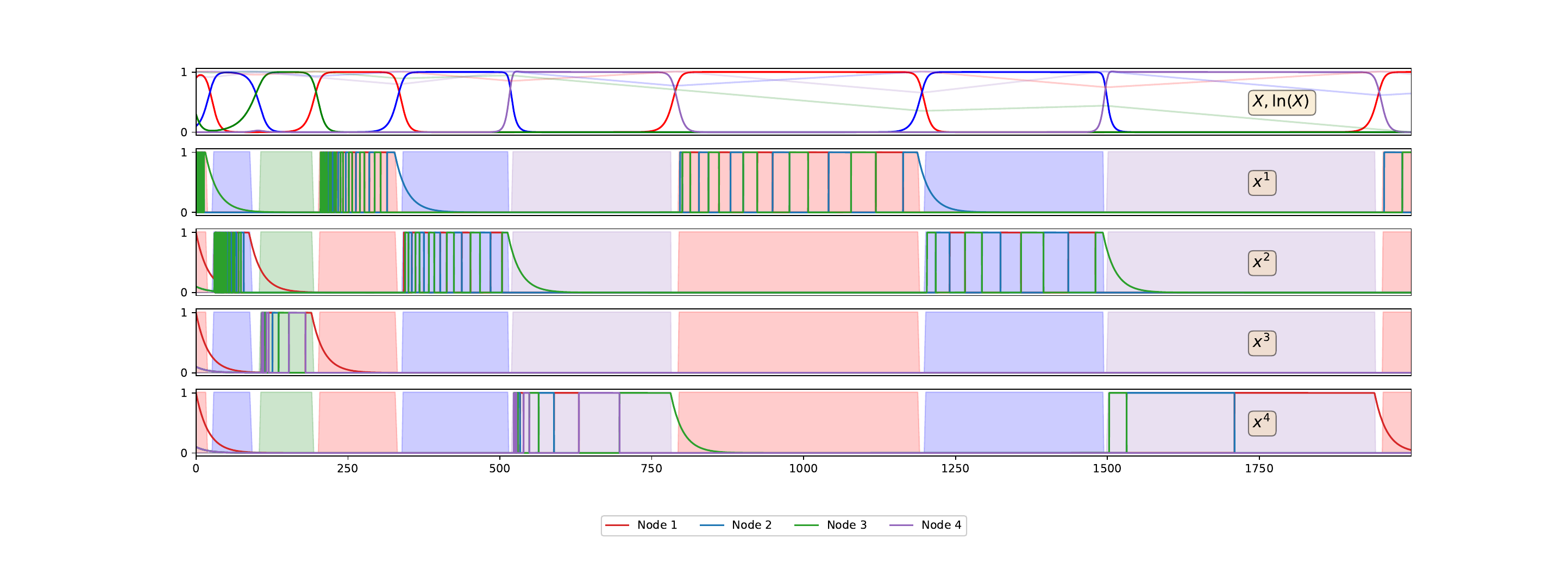}
    \caption{Timeseries of the realization of the digraph in \Cref{fig:ks-super}. The top panel shows the dynamics of the superstructure variables $X$. The four lower panels show the dynamics of the four substructure systems with coordinates $x^1, x^2, x^3, x^4$, respectively. The colored lines in all five panels correspond to the $3$, respectively $4$ components of the vectors and are color coded according to the legend. The colored overlay in the lower four panels indicates which of the subsystems is active, i.e., which equilibrium the $X$ component is close to. A full discussion including parameter values and interpretation can be found in \Cref{subsec:ks-super}.}
    \label{fig:ks-super-timeseries}
\end{figure}
\end{landscape}
\paragraph{Numerical simulation.}
We compute a trajectory for the initial condition
\begin{equation*}
    X = \begin{pmatrix}
        0.9 \\
        0.1 \\
        0.3 \\
        0.000001
    \end{pmatrix}, \quad x^1 = x^2 = \begin{pmatrix}
        0.999 \\
        0.1 \\
        0.1
    \end{pmatrix}, \quad x^3 = x^4 =\begin{pmatrix}
        0.999 \\
        0.1 \\
        0.1 \\
        0.1
    \end{pmatrix}.
\end{equation*}
\Cref{fig:ks-super-timeseries} contains the timeseries plot until $t=2000$. The top panel shows the dynamics of the superstructure variables $X$ (more precisely, each colored line displays the evolution of one of its components, color coded as in the legend. Similar to before, it displays the transition from the upper to the lower cycle according to the Kirk-Silber network
\begin{align*}
    \begin{pmatrix}
    	X_1 \\ X_2 \\ X_3 \\ X_4
    \end{pmatrix} \quad \approx \quad &\begin{pmatrix}
        1 \\ 0 \\ 0 \\ 0
    \end{pmatrix} \to 
    \begin{pmatrix}
        0 \\ 1 \\ 0 \\ 0
    \end{pmatrix} \to 
    \begin{pmatrix}
        0 \\ 0 \\ 1 \\ 0
    \end{pmatrix} \to 
    \begin{pmatrix}
        1 \\ 0 \\ 0 \\ 0
    \end{pmatrix} \\[5pt]
    &\to 
    \begin{pmatrix}
        0 \\ 1 \\ 0 \\ 0
    \end{pmatrix} \to 
    \begin{pmatrix}
        0 \\ 0 \\ 0 \\ 1
    \end{pmatrix} \to 
    \begin{pmatrix}
        1 \\ 0 \\ 0 \\ 0
    \end{pmatrix} \to
    \begin{pmatrix}
        0 \\ 1 \\ 0 \\ 0
    \end{pmatrix} \to 
    \begin{pmatrix}
        0 \\ 0 \\ 0 \\ 1
    \end{pmatrix} \to 
    \begin{pmatrix}
        1 \\ 0 \\ 0 \\ 0
    \end{pmatrix} \to\dotsb.
\end{align*}
Note again that the trajectory never actually reaches any of these points, but spends increasingly longer times in their vicinity every time it approaches them. As before, this can be seen in the shaded, logarithmically transformed timeseries.

The lower four panels of \Cref{fig:cyclic-super-timeseries} show the timeseries of the four subsystems in variables $x^1, x^2, x^3, x^4$ (which are $3$-, $3$-, $4$-, and $4$-dimensional, respectively) using the color code in the legend. Again, each of the subsystems is only ``active'', when the $X$ trajectory is near the corresponding equilibrium. When a subsystem is inactive, all its components decay to zero. Zooming into the subsystem panels (as is exemplified in \Cref{fig:ks-super-timeseries-zoom}) unveils details of the respective dynamics. One can see that subsystem $1$ in panel~(a) follows the cycle
\begin{equation*}
    \begin{pmatrix}
    	x^1_1 \\ x^1_2 \\ x^1_3
    \end{pmatrix} \quad \approx \quad \begin{pmatrix}
        1 \\ 0 \\ 0
    \end{pmatrix} \to \begin{pmatrix}
        0 \\ 1 \\ 0
    \end{pmatrix} \to \begin{pmatrix}
        0 \\0 \\ 1
    \end{pmatrix} \to 
    \begin{pmatrix}
        1 \\ 0 \\ 0
    \end{pmatrix} \to \dotsb.
\end{equation*}
Subsystem $2$ in panel~(b) follows the opposite cycle
\begin{equation*}
    \begin{pmatrix}
    	x^2_1 \\ x^2_2 \\ x^2_3
    \end{pmatrix} \quad \approx \quad \begin{pmatrix}
        1 \\ 0 \\ 0
    \end{pmatrix} \to \begin{pmatrix}
        0 \\ 0 \\ 1
    \end{pmatrix} \to \begin{pmatrix}
        0 \\1 \\ 0
    \end{pmatrix} \to 
    \begin{pmatrix}
        1 \\ 0 \\ 0
    \end{pmatrix} \to \dotsb.
\end{equation*}
Similarly, subsystem $3$ in panel~(c) follows the $4$-cycle
\begin{equation*}
    \begin{pmatrix}
    	x^3_1 \\ x^3_2 \\ x^3_3 \\ x^3_4
    \end{pmatrix} \quad \approx \quad \begin{pmatrix}
        1 \\ 0 \\ 0 \\ 0
    \end{pmatrix} \to \begin{pmatrix}
        0 \\ 1 \\ 0 \\ 0
    \end{pmatrix} \to \begin{pmatrix}
        0 \\ 0 \\ 1 \\ 0
    \end{pmatrix} \to  \begin{pmatrix}
        0 \\ 0 \\ 0 \\ 1
    \end{pmatrix} \to 
    \begin{pmatrix}
        1 \\ 0 \\ 0 \\ 0
    \end{pmatrix} \to \dotsb.
\end{equation*}
Subsystem $4$ in panel~(d) follows the same $4$-cycle in opposite direction.
\begin{equation*}
    \begin{pmatrix}
    	x^4_1 \\ x^4_2 \\ x^4_3 \\ x^4_4
    \end{pmatrix} \quad \approx \quad \begin{pmatrix}
        1 \\ 0 \\ 0 \\ 0
    \end{pmatrix} \to  \begin{pmatrix}
        0 \\ 0 \\ 0 \\ 1
    \end{pmatrix} \to  \begin{pmatrix}
        0 \\ 0 \\ 1 \\ 0
    \end{pmatrix} \to \begin{pmatrix}
        0 \\ 1 \\ 0 \\ 0
    \end{pmatrix} \to
    \begin{pmatrix}
        1 \\ 0 \\ 0 \\ 0
    \end{pmatrix} \to \dotsb.
\end{equation*}
Summarizing, the simulation displays the desired behavior that was proven to exist in \Cref{thm:simplex_simplex}: A driving superstructure reflecting $\GG$ and four subsystems that behave according to $\G_1,\G_2,\G_3,\G_4$.

\begin{figure}
    \centering
    \begin{overpic}[width=\linewidth]{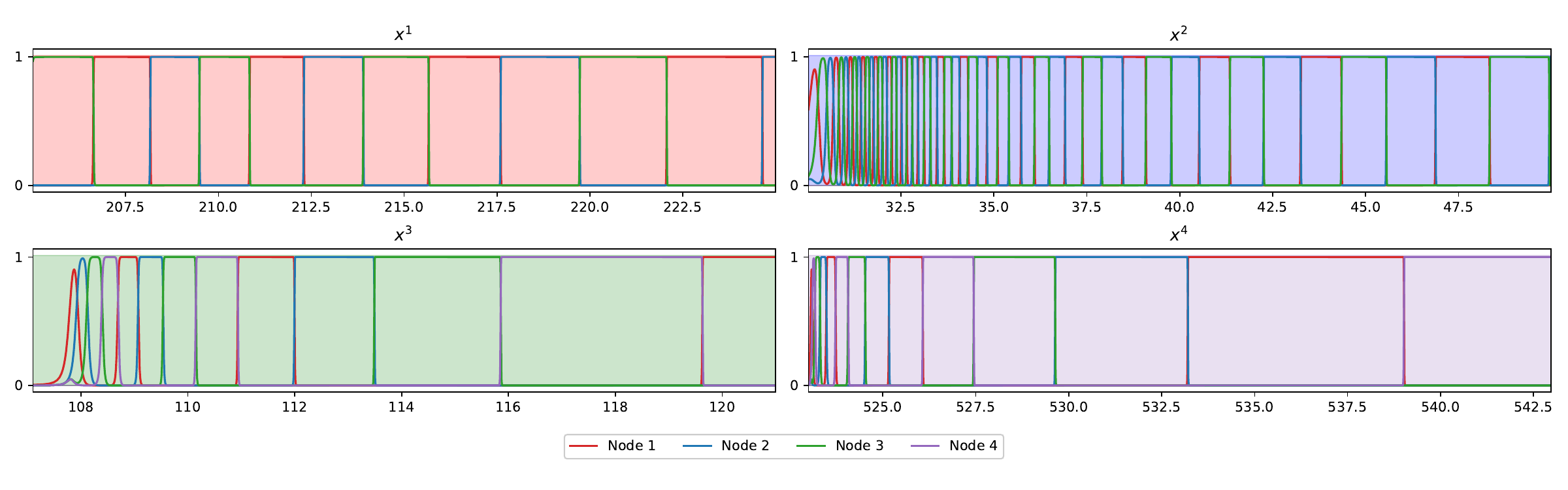}
    	\put(2,29){(a)}
    	\put(52,29){(b)}
    	\put(2,16){(c)}
        \put(52,16){(d)}
    \end{overpic}
    \caption{Zooming into substructure dynamics of \Cref{fig:ks-super-timeseries} for suitably chosen timeranges. (a) Subsystem $1$ displays cyclic behavior between three states. (b) Subsystem $2$ displays cyclic behavior between three states in the opposite direction. (c) Subsystem $3$ displays cyclic behavior between four states. (d) Subsystem $4$ displays cyclic behavior between four states in the opposite direction.}
    \label{fig:ks-super-timeseries-zoom}
\end{figure}
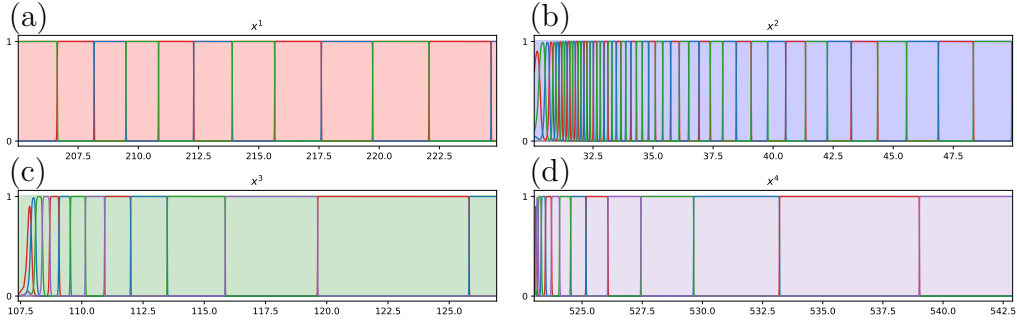

%%%
\subsection{Remarks about the numerics}
All of the numerical simulations discussed above have been performed in Python 3.12.1. The \texttt{numpy}-package is used for data storing. Crucially, trajectories are simulated via the \texttt{solve\_ivp}-routine from the \texttt{scipy}-package employing the RK45 (Runge-Kutta 45) method with tolerances \texttt{rtol=1e-12, atol=1e-12}. To avoid numerical instabilities, the simulation is performed in logarithmically transformed coordinates, which also requires the implementation of a suitably transformed vector field. For reproduction purposes, the code, including plotting commands, has been made available in an open repository \cite{vonderGrachtLohse.2026}. The simulations were run on a local personal computer and did not take more than a few minutes. After running the main script, the routines for simulation and timeseries plotting can be run from the console.

%%%%
\section{Discussion}
\label{sec:discussion}
%%%
\subsection{Summary and general comments}
In this article, we have presented a method for realizing a collection of directed graphs as a network in the phase space of a dynamical system, with excitable connections between different heteroclinic networks. As such, the dynamics exhibits transitions on two hierarchical levels. Our construction method is general in the sense that any $1$- and $2$-cycle free digraph can be realized on either hierarchical level. It is based on a suitable coupling of multiple realizations of individual digraphs according to the simplex method presented in \cite{AshPos2013}. We further discussed how different time-scales can be modified by adapting parameters, which is particularly significant for experimental observations of the desired transitions.

The hierarchical transition structures in our systems cannot only be regarded as dynamical transitions between realizations of \emph{different digraphs}.
Instead, they can be seen as a realization of dynamical transitions between \emph{different interaction structures on the same set of vertices}. In this interpretation, our construction yields a dynamical realization of a temporal network, i.e., a network with temporally changing interaction structure. 
We believe that this construction method can be a useful modeling tool for the description and investigation of complex systems with competing modulating mechanisms, such as competition scenarios in which competitors drop out of the game driven by some external mechanism.

The simplex method is a key element of our construction. In its original form, it generates heteroclinic connections between equilibria. Our generalization to invariant sets destroys the convergence in backwards time, as outlined above. Nonetheless, in forward time the dynamical behavior is identical to that of a heteroclinic network. Indeed, our numerical simulations suggest that the simulated trajectories behave in a way that is consistent with the presence of a depth two heteroclinic connection in forward time. We conjecture that this is the case for realizations of other collections of digraphs as well---the precise stability properties of the respective substructure realizations $\Net_j$ are crucial here. However, as already pointed out in \cite{AshPos2013}, no general results on stability and attraction of networks obtained through the simplex method are known.

%%%
\subsection{Further construction methods}
Our realization method suggests several angles of generalization and adaptation. In particular, the mechanism that relates super- and substructure realizations is reasonably flexible. Below we briefly and heuristically discuss some possibilities to replace parts of the construction to generalize our results.

\paragraph{Employing different realization methods.}
Our method uses the simplex method from~\cite{AshPos2013} for realizing both the super- and the substructure. Other realization methods exist that could be coupled in a similar way to obtain different hierarchical networks. For example, also in~~\cite{AshPos2013}, the authors propose the \emph{cylinder method} in which all equilibria are located on a single coordinate axis and each connection occurs within its own coordinate plane. This method can be used to replace one or both (su\-per-/sub\-structure) parts of our hierarchical realization. As long as the general structure of \eqref{eq:simplex_simplex}, consisting of a driving system and coupling modulated by suitable bump functions, is preserved, the global realization mechanism remains the same. That is, a very similar proof to the one of \Cref{thm:simplex_simplex} should show that the hierarchical structure can be realized as an excitable network with either or both of super- and substructure being realized by the cylinder method.

\paragraph{Hierarchical connections between other invariant sets.}
There are many cases in the literature where heteroclinic connections have been studied between invariant sets other than equilibria, such as periodic orbits and even chaotic sets, see e.g.~\cite{LohseRod2017} and~\cite{Ashwin.1998}, respectively. The mechanism underlying our hierarchical construction method is also not restricted to the case where the dynamics on the lower hierarchical level are heteroclinic networks between equilibria. Instead, the superstructure could modulate dynamical transitions between a collection of arbitrary invariant sets. To that end, one has to replace the non-trivial component of the convex combination in \eqref{eq:simplex_simplex_b} by a suitable equation that generates the desired invariant set. Likely, parameters would have to be carefully tuned to ensure that both the invariant substructures as well as the superstructure modulations will be observable in an actual realization.

\paragraph{More levels of hierarchy.}
In the same spirit as the previous paragraph, our construction idea can be generalized in an iterative manner: Multiple hierarchical realizations as in \eqref{eq:simplex_simplex}, each in their own respective coordinate space, could be considered as the invariant sets of a driving superstructure which are then coupled by a similar driver plus bump function construction as before. This realizes a three level hierarchical collection of digraphs as an excitable network of excitable networks of heteroclinic networks. This iterative process can in principle be generalized to an arbitrary number of hierarchy levels.

\paragraph{Adaptation to heteroclinic connections only.}
The main reason why the proposed realization method generates excitable connections with threshold $0$ instead of heteroclinic connections stems from the techniques in the substructure realization in \eqref{eq:simplex_simplex_b}. Therein, the second part of the convex combination controls the behavior of the substructure variables $x_i^j$ when the $j$-th substructure is inactive. It causes these variables to become unbounded in backwards time along trajectories which in forward time converge to $\Net_j$. If such behavior is undesired from a modeling perspective, this part of the convex combination can be modified as well. For example, replacing the $-x^j_j$ term by $x_i^j(x_i^j-1)$ causes the respective variable to be restricted to the compact interval $[0,1]$ for suitable initial conditions that guarantee convergence to $\Net_j$ in forward time. More precisely, the modified system
\begin{subequations}
	\begin{align}
	    \dot X_j &= X_j \left(1-\|X\|^2+\sum_{k=1}^N a_{kj}X_{k}^2 \right) \\[5pt]
	    \dot x^j_i &= x^j_i \left(\left( 1 - \|x^j\|^2 + \sum_{k=1}^{n_j} \alpha^j_{ki}(x^j_k)^2 \right) b^j_\varepsilon(X) - (1-b^j_\varepsilon(X)) \textcolor{red}{(1-x_i^j)} \right)
	\end{align}
\end{subequations}
(with all changes compared to \eqref{eq:simplex_simplex} highlighted in red) causes the substructure variables $x_i^j$ to decay to $0$ in forward time when $\Net_j$ is not active (as before), but to converge to $1$ in negative time when $\Net_j$ is not active. Thus, the connecting excitable trajectories discussed in the proof of \Cref{thm:simplex_simplex} do not have empty $\alpha$-limit sets. In particular, if we extend the definition of the invariant substructures $\Net_j$ to allow for $x_i^k\in\{0,1\}$ (instead of only $0$), these are dynamically invariant for the modified system and there are heteroclinic connections between them according to the superstructure $\GG$. This does, however, come at the cost of having the substructures consist of multiple copies of each individual network in different regions of phase space. Nonetheless, if this cost is accepted, one obtains proper heteroclinic connections.

%%%
\subsection{Outlook}
A key challenge in deriving our hierarchical construction method lies in the fact that the classical realization methods for digraphs as heteroclinic networks do not suggest a canonical way to replace equilibria by more general invariant sets. Instead, both the equilibria and the dynamical transitions depend crucially on the specific choice of coordinates and governing functions. Our method is therefore natural in the sense that, informally speaking, it replaces each coordinate axis by an individual subspace and each equilibrium by an invariant set (which is itself a digraph realization) in the respective subspace. This makes it an intuitive tool that should be applicable to various real-world hierarchical systems. 

However, the resulting system is potentially high-dimensional and some of the connections are excitable instead of heteroclinic. We believe that it is also possible to realize a hierarchical collection of digraphs in a lower-dimensional system with truly heteroclinic connections and without having multiple copies of the invariant sets as discussed above---at the cost of some restrictions in applicability. We will expand on this idea in a follow-up paper.

%%%%
\section*{Acknowledgement}
SvdG was partially funded by the Deutsche Forschungsgemeinschaft (DFG, German Research Foundation)—453112019.

%%%%
\printbibliography

\end{document}

%% file: img/ks-network.tex
	\centering
	\begin{tikzpicture}[
		->,
		>=open triangle 60,
		shorten >=1pt,
		shorten <=1pt,
		node/.style={
			circle,
			draw,
			fill = black,
			inner sep=2pt
		}
		]
		
		\node[node] (v31) at (-1,0) {};
		\node[node] (v32) at (1,0) {};
		\node[node] (v33) at (0,1) {};
		\node[node] (v34) at (0,-1) {};
		
		\path[>=stealth', line width=1pt, thick]
		(v31) edge (v32)
		(v32) edge (v33)
		(v32) edge (v34)
		(v33) edge (v31)
		(v34) edge (v31)
		;
		;
	\end{tikzpicture}

%% file: img/cyclic-superstructure.tex
	\centering
	\begin{tikzpicture}[
		->,
		>=open triangle 60,
		shorten >=1pt,
		shorten <=1pt,
		node/.style={
			circle,
			draw,
			fill = black,
			inner sep=2pt
		},
		quiver node/.style={
			color=blue!50, 
			line width=2pt, 
			rectangle,
			rounded corners, 
			draw, 
			font=\sffamily\bfseries, 
			inner sep=3pt, 
			minimum height=2cm, 
		}
	]
	
	\node[node] (v11) at (-1,0) {};
	\node[node] (v12) at (1,0) {};
	\node[node] (v13) at (0,1) {};
	
	\begin{scope}[shift={(6,0)}]
		\node[node] (v21) at (-1,0) {};
		\node[node] (v22) at (1,0) {};
		\node[node] (v23) at (0,1) {};
	\end{scope}
	
	\begin{scope}[shift={(3,4)}]
		\node[node] (v31) at (-1,0) {};
		\node[node] (v32) at (1,0) {};
		\node[node] (v33) at (0,1) {};
		\node[node] (v34) at (0,-1) {};
	\end{scope}
	
	%Substructures
	\path[>=stealth', line width=1pt, thick]
		(v11) edge (v12)
		(v12) edge (v13)
		(v13) edge (v11)
		
		(v21) edge (v23)
		(v22) edge (v21)
		(v23) edge (v22)
		
		(v31) edge (v32)
		(v32) edge (v33)
		(v32) edge (v34)
		(v33) edge (v31)
		(v34) edge (v31)
	;
	
	%Superstructure
	\node[quiver node, fit=(v11)(v12)(v13)] (q1) {};
	\node[quiver node, fit=(v21)(v22)(v23)] (q2) {};
	\node[quiver node, fit=(v31)(v32)(v33)(v34)] (q3) {};
	
	\path[blue!50, line width=2pt]
		(q1) edge [bend right] (q2)
		(q2) edge [bend right] (q3)
		(q3) edge [bend right] (q1)
	;
\end{tikzpicture}

%% file: img/ks-superstructure.tex
	\centering
	\begin{tikzpicture}[
		->,
		>=open triangle 60,
		shorten >=1pt,
		shorten <=1pt,
		node/.style={
			circle,
			draw,
			fill = black,
			inner sep=2pt
		},
		quiver node/.style={
			color=blue!50, 
			line width=2pt, 
			rectangle,
			rounded corners, 
			draw, 
			font=\sffamily\bfseries, 
			inner sep=3pt, 
			minimum height=2cm, 
		}
		]
		
		\node[node] (v11) at (-1,-.5) {};
		\node[node] (v12) at (1,-.5) {};
		\node[node] (v13) at (0,.5) {};
		
		\begin{scope}[shift={(8,0)}]
			\node[node] (v21) at (-1,-.5) {};
			\node[node] (v22) at (1,-.5) {};
			\node[node] (v23) at (0,.5) {};
		\end{scope}
		
		\begin{scope}[shift={(4,3)}]
			\node[node] (v31) at (-.75,-.75) {};
			\node[node] (v32) at (.75,-.75) {};
			\node[node] (v33) at (.75,.75) {};
			\node[node] (v34) at (-.75,.75) {};
		\end{scope}
		
		\begin{scope}[shift={(4,-3)}]
			\node[node] (v41) at (-.75,-.75) {};
			\node[node] (v42) at (.75,-.75) {};
			\node[node] (v43) at (.75,.75) {};
			\node[node] (v44) at (-.75,.75) {};
		\end{scope}
		
		%Substructures
		\path[>=stealth', line width=1pt, thick]
		(v11) edge (v12)
		(v12) edge (v13)
		(v13) edge (v11)
		
		(v21) edge (v23)
		(v22) edge (v21)
		(v23) edge (v22)
		
		(v31) edge (v32)
		(v32) edge (v33)
		(v33) edge (v34)
		(v34) edge (v31)
		
		(v41) edge (v44)
		(v42) edge (v41)
		(v43) edge (v42)
		(v44) edge (v43)
		;
		
		%Superstructure
		\node[quiver node, fit=(v11)(v12)(v13)] (q1) {};
		\node[quiver node, fit=(v21)(v22)(v23)] (q2) {};
		\node[quiver node, fit=(v31)(v32)(v33)(v34)] (q3) {};
		\node[quiver node, fit=(v41)(v42)(v43)(v44)] (q4) {};
		
		\path[blue!50, line width=2pt]
		(q1) edge [] (q2)
		(q2) edge [bend right] (q3)
		(q2) edge [bend left] (q4)
		(q3) edge [bend right] (q1)
		(q4) edge [bend left] (q1)
		;
	\end{tikzpicture}